\documentclass[a4paper,12pt]{article}

\usepackage{amscd}
\usepackage{amsfonts}
\usepackage{amsmath}
\usepackage{amssymb}
\usepackage{amsthm}
\usepackage[T1]{fontenc} % changing encode
\usepackage{here} % [h] for tables
\usepackage{mathrsfs} % \mathscr
\usepackage{txfonts} % colneqq
\usepackage[all]{xy} % xypic
\usepackage{algorithm} % pseudocode
\usepackage{algpseudocode} %pseudocode
\usepackage{diagbox} % \diagbox
\usepackage{listings} %lstlisting

\allowdisplaybreaks

\theoremstyle{plain}
\newtheorem{thm}{Theorem}[section]
\newtheorem{lmm}[thm]{Lemma}
\newtheorem{prp}[thm]{Proposition}
\newtheorem{crl}[thm]{Corollary}

\theoremstyle{definition}
\newtheorem{dfn}[thm]{Definition}

\newtheorem{exm}[thm]{Example}

\newcommand{\vs}[1][0.2]{\vspace{#1in}\noindent\ignorespaces}
\newcommand{\ba}{\begin{array*}}
\newcommand{\ea}{\end{array*}}
\newcommand{\be}{\begin{eqnarray*}}
\newcommand{\ee}{\end{eqnarray*}}
\newcommand{\bi}{\begin{itemize}}
\newcommand{\ei}{\end{itemize}}
\newcommand{\bb}{\vs\begin{itembox}}
\newcommand{\eb}{\end{itembox}}
\newcommand{\bc}{\begin{center}}
\newcommand{\ec}{\end{center}}
\newcommand{\bs}{\vs\begin{screen}}
\newcommand{\es}{\end{screen}}

\def\ens#1{{\mathchoice{\left\{ #1 \right\}}{\{ #1 \}}{\{ #1 \}}{\{ #1 \}}}}
\def\set#1#2{{\mathchoice{\left\{ #1 \ \middle| \ #2 \right\}}{\{ #1 \mid #2 \}}{\{ #1 \mid #2 \}}{\{ #1 \mid #2 \}}}}
\def\r#1{\text{\rm #1}}

\def\Bigv#1{\left| #1 \right|}
\def\v#1{{\mathchoice{\Bigv{#1}}{| #1 |}{| #1 |}{| #1 |}}}
\def\Bign#1{\left\| #1 \right\|}
\def\n#1{{\mathchoice{\Bign{#1}}{\| #1 \|}{\| #1 \|}{\| #1 \|}}}

\def\ol#1{\overline{#1}{}}

\newcommand{\bC}{\mathbb{C}}

\newcommand{\bN}{\mathbb{N}}

\newcommand{\bR}{\mathbb{R}}

\newcommand{\bZ}{\mathbb{Z}}

\newcommand{\cF}{\mathscr{F}}

\newcommand{\cO}{\mathscr{O}}
\newcommand{\cP}{\mathscr{P}}

\newcommand{\cU}{\mathscr{U}}

\newcommand{\rC}{\r{C}}

\newcommand{\rM}{\r{M}}

\newcommand{\C}{\bC}

\newcommand{\N}{\bN}

\newcommand{\R}{\bR}
\newcommand{\Z}{\bZ}

\newcommand{\Cp}{\mathbb{C}_p}

\newcommand{\Qp}{\mathbb{Q}_p}

\newcommand{\Zp}{\mathbb{Z}_p}

\newcommand{\ch}{\r{ch}}

\newcommand{\ev}{\r{ev}}

\newcommand{\GL}{\r{GL}}

\newcommand{\Hom}{\r{Hom}}

\newcommand{\im}{\r{im}}

\newcommand{\Mod}{\r{Mod}}

\newcommand{\pr}{\r{pr}}

\newcommand{\SL}{\r{SL}}

\newcommand{\Sym}{\r{Sym}}

\newcommand{\ZFC}{\textsf{ZFC}}

\newcommand{\WC}{\r{WC}}

\algnewcommand\algorithmicbreak{{\bf break}}
\algnewcommand\Break{\algorithmicbreak{}}
\algnewcommand\algorithmiccontinue{{\bf continue}}
\algnewcommand\Continue{\algorithmiccontinue{}}

\title{Schneider--Teitelbaum Duality over a Non-spherically Complete Field}
\author{Tomoki Mihara}
\date{}

\begin{document}

\maketitle
%\address
\begin{abstract}
We formulate Schneider--Teitelbaum duality between wide classes of Banach $k$-linear representations of $G$ and left $O_k[[G]]$-modules for a non-spherically complete field $k$, e.g.\ $\Cp$, and a profinite group $G$. We interpret a topological notion of a weak variant of irreducibility of a Banach $k$-linear representation of $G$ into a purely algebraic notion of a certain simplicity of the dual left $O_k[[G]]$-module. As applications, we give two $p$-adic families of infinite dimensional Banach $\Cp$-linear representations of a $p$-adic Lie group satisfying the weak irreducibility.
\end{abstract}

\tableofcontents
%\fn{}{}

\section{Introduction}
\label{Introduction}

We formulate a new duality as a partial extension of Schneider--Teitelbaum duality and also as an application of the non-Archimedean counterpart of Specker phenomenon. We give brief introduction to Schneider--Teitelbaum duality and Specker phenomenon together with preceding studies.

\vs
Let $k$ be a complete valuation field with a non-trivial valuation. We denote by $O_k$ the valuation ring of $k$. W.\ H.\ Schikhof introduced a contravariant categorical equivalence between Banach $k$-vector spaces and embeddable absolutely convex complete edged compactoids over $O_k$ in \cite{Sch95} Theorem 4.6. The abstract equivalence restricted to the case where $k$ is a local field, which we call {\it Schikhof duality}, is realised as the continuous dual functors between Banach $k$-vector spaces and compact Hausdorff flat linear topological $O_k$-modules (cf.\ \cite{ST02} Theorem 1.2 and \cite{Mih21-1} Proposition 1.7).

\vs
The key properties of Banach $k$-vector spaces and compact Hausdorff flat linear topological $O_k$-modules for this case is that every Banach $k$-vector space is isomorphic to the completed direct sum $\rC_0(I,k)$ and every compact Hausdorff flat linear topological $O_k$-module is isomorphic to the direct product $O_k^I$ for some set $I$. Namely, Schikhof duality interprets the completed direct sum and the direct product into each other.

\vs
There have ever been various extensions of Schikhof duality. First, a Banach $k$-vector space is naturally identified with a Banach $k$-linear representation of the trivial group $\ens{*}$, and a compact Hausdorff flat linear topological $O_k$-module is naturally identified with a compact Hausdorff flat linear topological $O_k[[\ens{*}]]$-module. P.\ Schneider and J.\ Teitelbaum extended Schikhof duality to a duality between Banach $k$-linear representations of a profinite group $G$ and compact Hausdorff flat linear topological $O_k[[G]]$-modules for the case where $k$ is a local field with $\ch(k) = 0$, which we call {\it Schneider--Teitelbaum duality}, in \cite{ST02} Theorem 2.3.

\vs
We extended Schikhof duality to dualities of several classes of locally convex $k$-vector spaces and linear topological $O_k$-modules in \cite{Mih21-1} Theorem 3.6, Theorem 3.20, and Theorem 3.32. Although there are many other duality theories on locally convex spaces (cf.\ \cite{Sch84} and \cite{Sch02}), our duality differs from them at the point that we considered an extension of Schikhof duality and dealt with linear topological $O_k$-modules.

\vs
Moreover, we extended Schikhof duality to specific symmetric monoidal categories in \cite{Mih21-1} Theorem 2.2 to obtain a duality between Abelian groups and rigid analytic Abelian groups as a non-Archimedean analogue of Pontryagin duality in \cite{Mih21-1} Theorem 3.5 and Theorem 3.16.

\vs
Further, we extended Schneider--Teitelbaum duality to unitarisable Banach $k$-linear representations of $G$ for the case where $G$ is a locally profinite group and $k$ is a local field without the restriction of $\ch(k)$ in \cite{Mih21-2} Theorem 3.17.

\vs
Our aim is to formulate Schneider--Teitelbaum duality for the case where $G$ is a profinite group and $k$ is a non-spherically complete field. We note that every local field is spherically complete, and the compactness of the continuous dual $O_k$-module of a Banach $k$-vector space does not hold for a complete valuation field which is not necessarily a local field. Therefore, in order to formulate Schneider--Teitelbaum duality for the case, we need to consider another class of $O_k$-modules than compact Hausdorff flat linear topological $O_k$-modules.

\vs
One candidate is the class of embeddable absolutely convex complete edged compactoids over $O_k$ as in the original abstract duality by Schikhof, but we chose another candidate: the class of (not topological) $O_k$-modules isomorphic to the direct product of copies of $O_k$ with a mild assumption of the cardinality of the index set because of the non-Archimedean counterpart of Specker phenomenon.

\vs
Now we focus on Specker phenomenon and related studies. To begin with, as an analogue of the non-reflexivity of $\ell^1(I,\C)$ and its dual Banach space $\ell^{\infty}(I,\C)$ for an infinite set $I$, the completed direct sum $\rC_0(I,k)$ and its dual Banach space $\ell^{\infty}(I,k)$, which is the localisation of the direct product $O_k^I$, are not reflexive if $k$ is spherically complete. More generally, as an analogue of the non-reflexivity of infinite dimensional vector spaces, if $k$ is spherically complete, a Banach $k$-vector space $V$ is reflexive if and only if $\dim_k V < \infty$ (cf.\ \cite{Roo78} 4.16).

\vs
However, if $k$ is non-spherically complete, then both of $\rC_0(I,k)$ and $\ell^{\infty}(I,k)$ are reflexive for any countable set $I$, and more generally, every Banach $k$-vector space of countable type and its dual Banach $k$-vector space are reflexive (cf.\ \cite{Roo78} 4.16 and 4.17). The result is extended to the case where $\# I$ is not $\aleph_1$-measurable (cf.\ \cite{Roo78} 4.21 and \cite{MN89} \S 12 Corollary 7.18). K.\ Eda further extended it to the case without restriction of $\# I$ (cf.\ \cite{MN89} \S 12 Theorem 7.17) in terms of projections associated to ultrafilters.

\vs
These results are non-Archimedean counterparts of studies of Abelian groups originated from Specker. First, {\it Specker's theorem} (cf.\ \cite{Spe50}) states that for any countable set $I$, the canonical morphism $\Z^{\oplus I} \to \Hom(\Z^I,\Z)$ is an isomorphism, and hence $\Z^{\oplus I}$ and $\Z^I$ are reflexive. Specker phenomenon refers to this pathological property of $\Z^I$ for a countable set $I$, which is introduced in \cite{Eda83} and is named by A.\ Blass in \cite{Bla92}, and has been deeply studied in various settings especially with slender groups and Fuchs-44-groups.

\vs
Specker's theorem is extended by A.\ Ehrenfeucht and J.\ {\L}o\'s in \cite{EL54} and by E.\ C.\ Zeeman \cite{Zee55} independently to what is nowadays called {\it {\L}o\'s's theorem}, which states that for any set $I$, if $\# I$ is not $\aleph_1$-measurable, then the canonical morphism $\Z^{\oplus I} \to \Hom(\Z^I,\Z)$ is an isomorphism, and hence $\Z^{\oplus I}$ and $\Z^I$ are reflexive.

\vs
K.\ Eda further extended {\L}o\'s's theorem in \cite{Eda82} Corollary 2 in terms of projections associated to ultrafilters. The combination of {\L}o\'s's theorem and Eda's theorem is called {\it {\L}o\'s--Eda theorem}.

\vs
Suppose that $k$ is non-spherically complete. The reflexivity of $\rC_0(I,k)$ and $\ell^{\infty}(I,k)$ for a countable set $I$ is a non-Archimedean counterpart of Specker's theorem. The reflexivity of $\rC_0(I,k)$ and $\ell^{\infty}(I,k)$ for a set $I$ such that $\# I$ is not $\aleph_1$-measurable is a non-Archimedean counterpart of {\L}o\'s's theorem. Eda's extension of the reflexivity of $\rC_0(I,k)$ and $\ell^{\infty}(I,k)$ in terms of projections associated to ultrafilters is a non-Archimedean counterpart of {\L}o\'s--Eda theorem. In this way, duality for Banach spaces over non-spherically complete field can be regarded as a non-Archimedean counterpart of studies of Specker phenomenon.

\vs
In this paper, we apply the reflexivity of $\rC_0(I,k)$ and $\ell^{\infty}(I,k)$ for a set $I$ such that $\# I$ is not $\aleph_1$-measurable to formulate a partial extension of Schneider--Teitelbaum duality to the case where $k$ is non-spherically complete. Unlike the original duality, we formulate a duality between a wide class of a Banach $k$-linear representation of a profinite group $G$ and a wide class of a (not topological) left $O_k[[G]]$-modules. Namely, our new duality interprets topological notions into purely algebraic notions.

\vs
In particular, we provide a criterion for a weak variant of the irreducibility of a Banach $k$-linear representation of $G$ in terms of a purely algebraic condition of the corresponding left $O_k[[G]]$-module for the case where $k$ is non-spherically complete, as an analogue of the interpretation of the irreducibility of a Banach $k$-linear representation of $G$ to the simplicity of a left $O_k[[G]]$-module for the case where $k$ is a local field of characteristic $0$ (cf.\ \cite{ST02} Corollary 3.6).

\vs
As applications, we give two $p$-adic families of infinite dimensional Banach $\Cp$-linear representations of a $p$-adic Lie group satisfying the weak irreducibility. One is a $p$-adic family interpolating symmetric products of the canonical representation of $\GL_2(\Zp)$ parametrised by weights in $O_{\Cp}$. Another one is a $p$-adic family given by continuous parabolic inductions of $\Cp$-valued characters of the diagonal torus of $\GL_2(\Zp)$.

\vs
We briefly explain contents of this paper. In \S \ref{Convention}, we introduce convention for this paper. In \S \ref{Dual Space}, we introduce the dual of a  Banach $k$-vector space as an $O_k$-module, and study its properties. In \S \ref{Schikhof Duality}, we formulate a partial extension of Schikhof duality for the case where $k$ is spherically complete, and show automatic continuity theorem for $O_k$-modules isomorphic to the direct product of copies of $O_k$ such that the cardinality of the index set is not $\aleph_1$-measurable. In \S \ref{Schneider--Teitelbaum Duality}, we formulate a partial extension of Schneider--Teitelbaum duality for the case where $k$ is spherically complete, and provide a criterion of the weak irreducibility in terms of the dual left $O_k[[G]]$-module. In \S \ref{Example}, we introduce two examples of $p$-adic families of weakly irreducible Banach $\Cp$-linear representations of a $p$-adic Lie group.

\section{Convention}
\label{Convention}

For a set $X$, we denote by $\# X$ its cardinality. We denote by $\N$ the set of non-negative integers, and set $\aleph_0 \coloneqq \# \N$. We denote by $\aleph_1$ the least uncountable cardinal. For a set $X$, we denote by $\cP_{< \aleph_0}(X)$ the set of finite subsets of $X$.

\vs
For sets $X$ and $Y$, we denote by $X^Y$ the set of maps $Y \to X$. When we handle a sequence or a family $v$ indexed by a set $I$, we frequently use the map notation $v(i)$ instead of the subscript notation $v_i$ to point the entry at $i \in I$, in order to avoid massive use of subscripts.

\vs
For a set $X$, an $x \in X$, and a binary relation $R$ on $X$, we set $X_{R x} \coloneqq \set{x' \in X}{x' R x}$. We note that every $d \in \N$ is set-theoretically identical to $(\aleph_0)_{< d}$. Therefore, for a set $X$, $X^d$ formally means $X^{(\aleph_0)_{< d}}$, which is naturally identified with the set of $d$-tuples in $X$.

\vs
For a set $X$ and a map $f \colon X \to \R_{\geq 0}$, we denote by $\sup_{x \in X} f(x)$ the supremum of the image of $f$ in $\R_{\geq 0} \sqcup \ens{\infty}$. In particular, $\sup_{x \in X} f(x)$ for the case $X = \emptyset$ is $0$ rather than $- \infty$ in our context.

\vs
A {\it complete valuation field} is a field $k$ equipped with a map $\v{\cdot} \colon k \to \R_{\geq 0}$ called a {\it (multiplicative) valuation} satisfying the following:
\bi
\item[(1)] For any $c \in k$, $\v{c} = 0$ holds if and only if $c = 0$ holds.
\item[(2)] For any $(c_0,c_1) \in k^2$, $\v{c_0 - c_1} \leq \max \ens{\v{c_0},\v{c_1}}$ holds.
\item[(3)] For any $(c_0,c_1) \in k^2$, $\v{c_0 c_1} = \v{c_0} \ \v{c_1}$ holds.
\item[(4)] The ultrametric on $k$ defined by
\be
k^2 & \to & \R_{\geq 0} \\
(c_0,c_1) & \mapsto & \v{c_0 - c_1}
\ee
is complete.
\ei
The reader should be careful not to confound the notations of the valuation $\v{\cdot}$ and the cardinality $\#$.

\vs
Throughout this paper, $k$ denotes a complete valuation field. We denote by $O_k$ the valuation ring $\set{c \in k}{\v{c} \leq 1}$ of $k$, and by $\Mod(O_k)$ the symmetric monoidal category of $O_k$-modules and the tensor product over $O_k$. We say that the valuation of $k$ is {\it non-trivial} if $\v{k} \neq \ens{0,1}$, and that $k$ is {\it non-spherically complete} if there exists a sequence of closed balls in $k$ satisfying the finite intersection property with empty intersection. If $k$ is non-spherically complete, then $\v{k}$ is dense in $\R_{\geq 0}$. A typical example of a non-spherically complete field is $\Cp$ for a prime number $p$. Henceforth, we always assume that the valuation of $k$ is non-trivial. 

\vs
A {\it Banach $k$-vector space} is a $k$-vector space $V$ equipped with a map $\n{\cdot} \colon V \to \R_{\geq 0}$ called a {\it norm} satisfying the following:
\bi
\item[(1)] For any $v \in V$, $\n{v} = 0$ holds if and only if $v = 0$ holds.
\item[(2)] For any $(v_0,v_1) \in V^2$, $\n{v_0 - v_1} \leq \max \ens{\n{v_0},\n{v_1}}$ holds.
\item[(3)] For any $(c,v) \in k \times V$, $\n{cv} = \v{c} \ \n{v}$ holds.
\item[(4)] The ultrametric on $V$ defined by
\be
V^2 & \to & \R_{\geq 0} \\
(v_0,v_1) & \mapsto & \n{v_0 - v_1}
\ee
is complete.
\ei
Let $V$ and $W$ be Banach $k$-vector spaces. A $k$-linear homomorphism $f \colon V \to W$ is said to be {\it bounded} if there exists a $C \in \R_{\geq 0}$ such that for any $v \in V$, $\n{f(v)} \leq C \n{v}$ holds, and is said to be {\it contracting} if such a $C$ can be chosen to be $1$. We denote by $\n{f}$ the infimum of such a $C$, and call it {\it the operator norm of $f$}. We denote by $\Hom(V,W)$ the Banach $k$-vector space of bounded $k$-linear homomorphisms $V \to W$ equipped with the operator norm, and by $\Hom_{\leq 1}(V,W)$ its subset of contracting $k$-linear homomorphisms.

\vs
Let $W$ be a closed $k$-linear subspace of a Banach $k$-vector space $V$. Then $W$ forms a Banach $k$-vector space with respect to the restriction of the structure of $V$, and the quotient $k$-vector space $V/W$ also forms a Banach $k$-vector space with respect to the quotient norm
\be
V/W & \to & \R_{\geq 0} \\
\ol{v} & \mapsto & \inf_{v \in \ol{v}} \n{v}.
\ee
We always regard $W$ and $V/W$ as Banach $k$-vector spaces in these ways.

\vs
Let $I$ be a set. We denote by $\ell^{\infty}(I,k)$ the Banach $k$-vector space of maps $v \colon I \to k$ with $\sup_{i \in I} \v{f(i)} < \infty$ equipped with the supremum norm
\be
\ell^{\infty}(I,k) & \to & \R_{\geq 0} \\
f & \mapsto & \sup_{i \in I} \v{f(i)},
\ee
and by $\rC_0(I,k) \subset \ell^{\infty}(I,k)$ the closed $k$-linear subspace given by
\be
\set{f \in \ell^{\infty}(I,k)}{\forall \epsilon \in \R_{> 0} \left[ \# \set{i \in I}{\n{v(i)} \geq \epsilon} < \aleph_0 \right]}.
\ee
For an $i \in I$, we abuse the convention $\delta_i$ without specifying $I$ for the characteristic function $I \to k$ of $\ens{i}$ regarded as an element of $\rC_0(I,k) \cap O_k^I \subset \ell^{\infty}(I,k)$.

\vs
For an $f \in \rC_0(I,k)$, we denote by $\sum_{i \in I} f(i) \in V$ the usual sum of $f$ if $\# I < \aleph_0$, and otherwise the unique $c \in k$ satisfying that for any $\epsilon \in \R_{> 0}$, there exists an $I_0 \in \cP_{< \aleph_0}(I)$ such that $\n{f(i)} < \epsilon$ holds for any $i \in I \setminus I_0$ and $\n{c - \sum_{i \in I_0} f(i)} < \epsilon$ holds.

\section{Dual Space}
\label{Dual Space}

Let $V$ be a Banach $k$-vector space. For a Banach $k$-vector space $V$, we denote by $V^{\circ}$ the underlying $O_k$-module of the closed unit ball $\set{v \in V}{\n{v} \leq 1}$. We set $V^* \coloneqq \Hom(V,k)$, and $V^{\vee} \coloneqq (V^*)^{\circ}$. We call $V^{\vee}$ {\it the dual of $V$}. We note that $V^{\vee}$ coincides with the underlying $O_k$-module of $\Hom_{\leq 1}(V,k)$. Since the valuation of $k$ is non-trivial, the inclusion $V^{\vee} \hookrightarrow V^*$ induces a $k$-linear isomorphism $k \otimes_{O_k} V^{\vee} \to V^*$.

\vs
Let $M$ be an $O_k$-module. We denote by $M^{\vee}$ the closed $k$-linear subspace of $\ell^{\infty}(M,k)$ consisting of elements which are $O_k$-linear homomorphisms. We call $M^{\vee}$ {\it the dual of $M$}. Since we abuse the convention of the dual for a Banach $k$-vector space and an $O_k$-module, we distinguish a Banach $k$-vector space and its underlying $O_k$-module.

\subsection{Banach Localisation and Weak Topology}
\label{Banach Localisation and Weak Topology}

Let $M$ be a torsion-free $O_k$-module. Since $M$ is torsion-free, the canonical $O_k$-linear homomorphism $M \to k \otimes_{O_k} M$ is injective, and hence we identify $M$ with its image in $k \otimes_{O_k} M$. Similarly, for any $O_k$-submodule $L \subset M$, we identify $k \otimes_{O_k} L$ with its image in $k \otimes_{O_k} M$.

\begin{dfn}
\label{saturated}
Let $M$ be an $O_k$-module. When $M$ is torsion-free, we denote by $\n{\cdot}_M$ the seminorm
\be
k \otimes_{O_k} M & \to & \R_{\geq 0} \\
m & \mapsto & \inf \set{\v{c}^{-1}}{c \in k^{\times} \land cm \in M}
\ee
on $k \otimes_{O_k} M$, and call it {\it the gauge associated to $M$}. We say that $M$ is {\it generically complete} if $M$ is torsion-free and $\n{\cdot}_M$ is a complete norm, in which case we denote by $kM$ the resulting Banach $k$-vector space $(k \otimes_{O_k} M,\n{\cdot})$ and call it {\it the Banach localisation of $M$}. We say that $M$ is an {\it saturated} if $M$ is generically complete and coincides with $(kM)^{\circ}$.
\end{dfn}

The seminorm $\n{\cdot}_M$ is called {\it the gauge}. We recall elementary properties of gauges well-known to experts.

\begin{prp}
\label{saturated closed unit ball}
Suppose that $\v{k}$ is dense in $\R_{\geq 0}$. Let $V$ be a Banach $k$-vector space. Then $V^{\circ}$ is saturated, and the scalar extension $k(V^{\circ}) \to V$ of the inclusion $V^{\circ} \hookrightarrow V$ is an isometric $k$-linear isomorphism.
\end{prp}

\begin{proof}
The assertion immediately follows from \cite{Sch02} Lemma 2.2 ii.
\end{proof}

\begin{crl}
\label{dual norm}
Suppose that $\v{k}$ is dense in $\R_{\geq 0}$. Then for any Banach $k$-vector space $V$, $V^{\vee}$ is saturated, and the scalar extension $k(V^{\vee}) \to V^*$ of the inclusion $V^{\vee} \hookrightarrow V^*$ is an isometric $k$-linear isomorphism.
\end{crl}

\begin{proof}
The assertion is a direct consequence of Proposition \ref{saturated closed unit ball} applied to $V^*$.
\end{proof}

\begin{crl}
\label{cofree implies saturated}
Suppose that $\v{k}$ is dense in $\R_{\geq 0}$. Then for any set $I$, $O_k^I$ is saturated.
\end{crl}

\begin{proof}
The assertion is a direct consequence of the first assertion of Proposition \ref{saturated closed unit ball} applied to $\ell^{\infty}(I,k)$.
\end{proof}

\begin{prp}
\label{operator norm}
Let $V$ be a Banach $k$-vector space. The restriction map $V^* \to (V^{\circ})^{\vee}$ is a $k$-linear isomorphism, and in addition if $\v{k}$ is dense in $\R_{\geq 0}$, then it is an isometry.
\end{prp}

\begin{proof}
For any bounded $O_k$-linear homomorphism $\psi \colon V^{\circ} \to k$, the scalar extension $k \otimes_{O_k} \psi \colon V \to k$ is bounded. This implies the first assertion. The second assertion follows from Proposition \ref{saturated closed unit ball} by the definition of the operator norm.
\end{proof}

For $O_k$-modules $M_0$ and $M_1$, we denote by $\Hom_{O_k}(M_0,M_1)$ the $O_k$-module of $O_k$-linear homomorphisms $M_0 \to M_1$. We have the following adjoint property of the Banach localisation:

\begin{prp}
\label{localisation operator norm}
Suppose that $\v{k}$ is dense in $\R_{\geq 0}$. Let $M$ be a saturated $O_k$-module and $V$ a Banach $k$-vector space. Then the restriction map $\Hom_{\leq 1}(kM,V) \to \Hom_{O_k}(M,V^{\circ})$ is an $O_k$-linear isomorphism.
\end{prp}

\begin{proof}
Since $M$ generates $kM$, it suffices to show the surjectivity. Let $\psi$ be an $O_k$-linear homomorphism $M \to V^{\circ}$. We denote by $k \psi$ the scalar extension $kM \to V$ of $\psi$. It suffices to show $\n{k \psi} \leq 1$. Assume that there exists an $m \in kM$ such that $\n{(k \psi)(m)} > \n{m}_M$. By $\n{(k \psi)(0)} = 0 = \n{0}_M$, we have $m \neq 0$ and hence $\n{m}_M > 0$. Since $\v{k}$ is dense in $\R_{\geq 0}$, there exists a $c \in k^{\times}$ with $\n{(k \psi)(m)}^{-1} < \v{c} \leq \n{m}_M^{-1}$. We have
\be
\n{cm}_M = \v{c} \ \n{m}_M \leq 1,
\ee
and hence $cm \in M$ by the saturatedness of $M$. In particular, we have
\be
(k \psi)(cm) = \psi(cm) \in \im(\psi) \subset V^{\circ}.
\ee
On the other hand, we also have
\be
\n{(k \psi)(cm)} = \v{c} \ \n{(k \psi)(m)} > 1,
\ee
which contradicts $(k \psi)(cm) \in V^{\circ}$. Therefore, we obtain $\n{k \psi} \leq 1$.
\end{proof}

\begin{crl}
\label{Banach vs saturated}
Suppose that $\v{k}$ is dense in $\R_{\geq 0}$. Then the correspondences $M \mapsto kM$ and $V \mapsto V^{\circ}$ define a $\Mod(O_k)$-enriched equivalence between the $\Mod(O_k)$-enriched category of Banach $k$-vector spaces and contracting $k$-linear homomorphisms and the $\Mod(O_k)$-enriched category of saturated $O_k$-modules and $O_k$-linear homomorphisms.
\end{crl}

\begin{proof}
The assertion immediately follows from Proposition \ref{saturated closed unit ball} and Proposition \ref{localisation operator norm}.
\end{proof}

We studies relation of norms and duals. We consider another algebraic topology than the norm topology, which is useful in the study of duals.

\begin{dfn}
{\it The weak topology} of a Banach $k$-vector space $V$ is the weakest topology on $V$ for which every $m \in V^{\vee}$ is a continuous map $V \to k$. {\it The weak topology} of an $O_k$-module $M$ is the weakest topology on $M$ for which every $v \in M^{\vee}$ is a continuous map $M \to k$.
\end{dfn}

Here, since we distinguish a Banach $k$-vector space and its underlying $O_k$-module, the dual of a Banach $k$-vector space is distinguished from the dual of its underlying $O_k$-module. The former one is a topological notion, while the latter one is a purely algebraic notion. When we refer to a notion as purely algebraic one, we intend that we have provided a definition only using an algebraic structure. For example, the canonical topology of a finite dimensional vector space over a complete valuation field is a purely algebraic notion.

\begin{dfn}
A {\it weakly closed} (resp.\ {\it weakly open}, {\it weakly dense}) subset means a closed (resp.\ open, dense) subset with respect to the weak topology. A {\it weakly convergent} net means a net convergent with respect to the weak topology. A {\it weakly continuous map} (resp.\ {\it weak homeomorphism}) means a map continuous (resp.\ homeomorphic) with respect to the weak topologies of the domain and the codomain.
\end{dfn}

\begin{exm}
\label{weak closedness vs closedness}
For any Banach $k$-vector space $V$, every contracting homomorphism $m \in V \to k$ is continuous, and hence every weakly closed subset of $V$ is closed. On the other hand, the converse does not necessarily hold. Take a $c \in k^{\times}$ with $\v{c} < 1$. The closure of the $k$-linear subspace of $\rC_0(\N,k)$ generated by $\set{\delta_0 - c^{-i} \delta_i}{i \in \N}$ is a proper subspace because $\delta_0$ does not belong to it, but is weakly dense because the sequence $(\delta_0 - c^{-i} \delta_i)_{i \in \N}$ weakly converges to $\delta_0$. In particular, a closed $k$-linear subspace of the Banach $k$-vector space $\rC_0(\N,k)$ is not necessarily weakly closed.
\end{exm}

We introduce two notions which play important roles when we extract an information of an $O_k$-submodule of a saturated $O_k$-module from that of the Banach localisation.

\begin{dfn}
An $O_k$-submodule $L$ of an $O_k$-module $M$ is said to be {\it pure} if the equality
\be
\set{m \in M}{\exists c \in O_k \setminus \ens{0}[cm \in L]} = L,
\ee
holds, and is said to be {\it generically weakly closed} in $M$ if $M$ is generically complete and $k \otimes_{O_k} L$ is weakly closed in $kM$.
\end{dfn}

\begin{prp}
\label{purity criterion}
Let $M$ be a torsion-free $O_k$-module.
\bi
\item[(1)] For any $O_k$-submodule $L \subset M$, $L$ is pure if and only if $L = (k \otimes_{O_k} L) \cap M$.
\item[(2)] For any $k$-linear subspace $W \subset kM$, $W \cap M$ is a pure $O_k$-submodule of $M$.
\ei
\end{prp}

\begin{proof}
(1) We have $L \subset (k \otimes_{O_k} L) \cap M$. First, assume $L = (k \otimes_{O_k} L) \cap M$. For any $m \in M$, if there exists a $c \in O_k \setminus \ens{0}$ such that $cm \in L$, then we have $m = c^{-1} cm \in k \otimes_{O_k} L$, and hence $m \in L$.

\vs
Next, assume that $L$ is pure. Let $m \in (k \otimes_{O_k} L) \cap M$. By $m \in k \otimes_{O_k} L$, there exists a $c \in O_k \setminus \ens{0}$ such that $cm \in L$. By $m \in M$ and the purity of $L$, we obtain $m \in L$.

\vs
(2) The assertion immediately follows from (1) by $W = k \otimes_{O_k} (W \cap M)$.
\end{proof}

We will give an example of a weak generically closed $O_k$-submodule in Proposition \ref{topology of perp} in the next subsection.

\subsection{Annihilator Subspace}
\label{Annihilatpr Subpace}

For a subset $W$ of a Banach $k$-vector space $V$, we set $W^{\perp} \coloneqq \set{m \in V^{\vee}}{m |_W = 0}$. For a subset $L$ of an $O_k$-module $M$, we set $L^{\perp} \coloneqq \set{v \in M^{\vee}}{v |_L = 0}$.

\begin{prp}
\label{topology of perp}
\bi
\item[(1)] For any subset $W$ of a Banach $k$-vector space $V$, $W^{\perp}$ is a weakly closed generically weakly closed pure $O_k$-submodule of $V^{\vee}$.
\item[(2)] For any subset $L$ of an $O_k$-module $M$, $L^{\perp}$ is a closed weakly closed $k$-linear subspace of $M^{\vee}$.
\ei
\end{prp}

\begin{proof}
(1) For a $v \in W$, we denote by $\ev_v$ the scalar extension $k(V^{\vee}) \to k$ of the evaluation map $V^{\vee} \to O_k$ at $v$. By Proposition \ref{localisation operator norm}, $\ev_v$ is contracting for any $v \in W$. We have
\be
k \otimes_{O_k} W^{\perp} & = & \set{m \in k(V^{\vee})}{\exists c \in k^{\times}[cm \in W^{\perp}]} \\
& = & \set{m \in k(V^{\vee})}{\exists c \in k^{\times}[cm \in V^{\vee} \land (cm) |_W = 0]} \\
& = & \set{m \in k(V^{\vee})}{\exists c \in k^{\times}[cm \in V^{\vee} \land \forall v \in W[(cm)(v) = 0]]} \\
& = & \set{m \in k(V^{\vee})}{\exists c \in k^{\times}[cm \in V^{\vee} \land \forall v \in W[\ev_v(cm) = 0]]} \\
& = & \set{m \in k(V^{\vee})}{\forall v \in W[\ev_v(m) = 0]}.
\ee
This implies that $W^{\perp}$ is generically weakly closed. We have
\be
(k \otimes_{O_k} W^{\perp}) \cap V^{\vee} = \set{m \in V^{\vee}}{\forall v \in W[\ev_v(m) = 0]} = \set{m \in V^{\vee}}{\forall v \in W[m(v) = 0]} = W^{\perp}.
\ee
This implies that weakly closed by Proposition \ref{cofree scalar extension} (2) and $W^{\perp}$ is pure by Proposition \ref{purity criterion} (1).

\vs
(2) We denote by $\iota \colon M \to M^{\vee \vee}$ the canonical $O_k$-linear homomorphism. We have
\be
L^{\perp} = \set{v \in M^{\vee}}{v |_L = 0} = \set{v \in M^{\vee}}{\forall m \in L[\iota(m)(v) = 0]},
\ee
and hence $L^{\perp}$ is weakly closed. By the first sentence of Example \ref{weak closedness vs closedness}, weak closedness implies closedness.
\end{proof}

\begin{prp}
\label{double perp}
\bi
\item[(1)] For any weakly closed $k$-linear subspace $W$ of a Banach $k$-vector space $V$, the equality $W = W^{\perp \perp}$ holds.
\item[(2)] If $\v{k}$ is dense in $\R_{\geq 0}$, then for any generically weakly closed pure $O_k$-submodule $L$ of a saturated $O_k$-module $M$, the equality $L = L^{\perp \perp}$ holds.
\ei
\end{prp}

In order to prove Proposition \ref{double perp} (1) and (2) in a common way, we prepare a lemma.

\begin{lmm}
\label{separating homomorphism}
Let $V$ be a Banach $k$-vector space, and $W$ a $k$-linear subspace of $V$. For any $(F,v) \in \cP_{< \aleph_0}(V^{\vee}) \times V$ such that $\sup_{m \in F} \v{m(v' - v)} > \epsilon$ for any $v' \in W$, there exists a contracting $k$-linear homomorphism $V \to k$ such that $\phi |_W = 0$ and $\phi(v) \neq 0$.
\end{lmm}

\begin{proof}
We denote by $\phi \colon V \to \ell^{\infty}(F,k)$ the contracting $k$-linear homomorphism associated to $F$. For any $v' \in W$, we have
\be
\n{\phi(v') - \phi(v)} = \n{\phi(v'-v)} = \sup_{m \in F} \v{m(v'-v)} > \epsilon.
\ee
This implies that $\phi(v) \notin \phi(W)$. Since $\ell^{\infty}(F,k) = \rC_0(F,k)$ is a finite dimensional Banach $k$-vector space admitting an orthonormal Schauder basis, there exists a contracting $k$-linear homomorphism $\phi' \colon \ell^{\infty}(F,k) \to k$ such that $\phi' |_{\phi(W)} = 0$ and $\phi'(\phi(v)) \neq 0$ (cf.\ \cite{BGR84} Proposition 2.4.1/5 or elementary linear algebra using $\Hom_{\leq 1}(\rC_0(F,k),\ell^{\infty}(F,k)) \cong \rM_{\# F}(O_k)$). We obtain $(\phi' \circ \phi)|_W = 0$ and $(\phi' \circ \phi)(v) \neq 0$.
\end{proof}

\begin{proof}[Proof of Proposition \ref{double perp}]
(1) We have $W \subset W^{\perp \perp}$ by the definition. It suffices to show $v \notin W^{\perp \perp}$ for any $v \in V \setminus W$. Since $W$ is weakly closed, there exists an $(F,\epsilon) \in \cP_{< \aleph_0}(V^{\vee}) \times \R_{> 0}$ such that $v' \notin W$ for any $v' \in V$ with $\sup_{m \in F} \v{m(v' - v)} \leq \epsilon$. Therefore we have $v \notin W^{\perp \perp}$ by Lemma \ref{separating homomorphism} applied to $(V,W,F,v)$.

\vs
(2) We have $L \subset L^{\perp \perp}$ by the definition. It suffices to show $m \notin L^{\perp \perp}$ for any $m \in M \setminus L$. Since $L$ is pure, we have $m \in kM \setminus (k \otimes_{O_k} L)$ by Proposition \ref{purity criterion} (1). Since $L$ is generically weakly closed, there exists an $(F,\epsilon) \in \cP_{< \aleph_0}((kM)^{\vee}) \times \R_{> 0}$ such that $v' \notin k \otimes_{O_k} L$ for any $v' \in kM$ with $\sup_{m \in F} \v{m(v' - v)} \leq \epsilon$. Therefore we have $m \notin L^{\perp \perp}$ by Proposition \ref{localisation operator norm} applied to $(M,k)$ and Lemma \ref{separating homomorphism} applied to $(kM,k \otimes_{O_k} L,F,m)$.
\end{proof}

\begin{crl}
\bi
\item[(1)] For any subset $W$ of a Banach $k$-vector space $V$, $W^{\perp \perp}$ is the weak closure in $V$ of the $k$-linear subspace generated by $W$.
\item[(2)] If $\v{k}$ is dense in $\R_{\geq 0}$, then for any subset $L$ of a saturated $O_k$-module $M$, $L^{\perp \perp}$ is the intersection of $M$ and the weak closure in $kM$ of the $k$-linear subspace generated by $L$.
\ei
\end{crl}

\begin{proof}
Since $\perp$ is anti-order-preserving, the assertion (1) follows from Proposition \ref{double perp} (1) and the assertion (2) follows from Proposition \ref{purity criterion} (2) and Proposition \ref{double perp} (2).
\end{proof}

For a Banach $k$-vector space $V$, we denote by $\WC(V)$ the set of weakly closed $k$-linear subspaces of $V$ directed by inclusion. For an $O_k$-module $M$, we denote by $\WC(M)$ the set of generically weakly closed pure $O_k$-submodules of $M$ directed by inclusion.

\begin{crl}
\label{perp equivalence}
Suppose that $\v{k}$ is dense in $\R_{\geq 0}$.
\bi
\item[(1)] For any Banach $k$-vector space $V$, the correspondences $\perp$ define anti-order-preserving maps between $\WC(V)$ and $\WC(V^{\vee})$ which are inverses to each other.
\item[(2)] For any saturated $O_k$-module $M$, the correspondences $\perp$ define anti-order-preserving maps between $\WC(M)$ and $\WC(M^{\vee})$ which are inverses to each other.
\ei
\end{crl}

\begin{proof}
The assertion (1) follows from the first assertion of Proposition \ref{dual norm} and Proposition \ref{double perp}, and the assertion (2) follows from Proposition \ref{double perp}.
\end{proof}

Through Corollary \ref{perp equivalence}, we interpret the topological notion of the weak closedness in a Banach $k$-vector space into the purely algebraic notion of the weak closedness in an $O_k$-module. We note that we can forget the norm of a Banach $k$-vector space by simply taking the closed unit ball rather than the dual, but the use of the dual is helpful because the weak topology of the dual is easy to handle when we consider representations of profinite groups. See \S \ref{Example} for actual computations.

\section{Schikhof Duality}
\label{Schikhof Duality}

We study the dual in \S \ref{Dual Space} for the case where $k$ is non-spherically complete to obtain a categorical equivalence between wide classes of Banach $k$-vector spaces and $O_k$-modules.

\subsection{Categorical Equivalence}
\label{Categorical Equivalence 1}

A filter $\cF$ is said to be {\it $\lambda$-complete} for a cardinal $\lambda$ if $\bigcap_{U \in S} U \in \cF$ holds for any subset $S \subset \cF$ with $0 < \# S < \lambda$. A cardinal $\kappa$ is said to be {\it $\lambda$-measurable} for a cardinal $\lambda$ if $\kappa \geq \aleph_1$ holds and $\kappa$ admits a $\lambda$-complete non-principal ultrafilter, and is said to be {\it measurable} if $\kappa$ is $\kappa$-measurable.

\vs
We note that a cardinal $\kappa$ is $\aleph_1$-measurable if and only if there exists a measurable cardinal not greater than $\kappa$. The existence of a measurable cardinal is unprovable under $\ZFC$ a long as $\ZFC$ is consistent, and we work in $\ZFC$ without the assumption of the existence of a measurable cardinal.

\begin{dfn}
A set is said to be {\it non-measurable} if its cardinality is not $\aleph_1$-measurable. A Banach $k$-vector space $V$ is said to be {\it free} if there exists an isomeric $k$-linear isomorphism $\rC_0(I,k) \to V$ for some non-measurable set $I$. An $O_k$-module $M$ is said to be {\it cofree} if there exists an $O_k$-linear isomorphism $M \to O_k^I$ for some non-measurable set $I$.
\end{dfn}

By Corollary \ref{cofree implies saturated}, every cofree $O_k$-module is saturated. We implicitly use this property when we apply results on a saturate $O_k$-module in \S \ref{Dual Space}. We will give non-trivial examples of cofree $O_k$-modules in Proposition \ref{universality of measure space} (1) and Proposition \ref{weak topology vs inverse limit topology}. We give an algebraic variant of Schikhof duality (cf.\ \cite{Sch95} Theorem 4.6) as a non-spherically complete analogue of \cite{ST02} Theorem 1.2 and \cite{Mih21-2} Theorem 2.2.

\begin{thm}[Schikhof duality for free Banach spaces and cofree modules]
\label{Schikhof}
\bi
\item[(1)] For any free Banach $k$-vector space $V$, $V^{\vee}$ is cofree, and in addition if $k$ is non-spherically complete, then the canonical $k$-linear homomorphism $V \to V^{\vee \vee}$ is an isometric isomorphism.
\item[(2)] If $k$ is non-spherically complete, then for any cofree $O_k$-module $M$, $M^{\vee}$ is free, and the canonical $O_k$-linear homomorphism $M \to M^{\vee \vee}$ is an isomorphism.
\ei
\end{thm}

\begin{proof}
(1) For any (possibly measurable) set $I$, the canonical pairing $O_k^I \times \rC_0(I,k) \to k$ induces an $O_k$-linear isomorphism $O_k^I \to \rC_0(I,k)^{\vee}$ by \cite{MN89} \S 12 Corollary 7.13. This implies the first assertion. By Proposition \ref{operator norm} applied to $V^*$, the restriction map $V^{**} \to V^{\vee \vee}$ is an isometric $k$-linear isomorphism. Therefore the second assertion follows from \cite{MN89} \S 12 Corollary 7.20.

\vs
(2) Take an $O_k$-linear isomorphism $M \to O_k^I$ with a non-measurable set $I$. Since $O_k^I$ is isomorphic to $\rC_0(I,k)^{\vee}$, we may assume $M = \rC_0(I,k)^{\vee}$. We have $M^{\vee} \cong \rC_0(I,k)^{\vee \vee} \cong \rC_0(k,I)$ by (1), and hence $M^{\vee}$ is free. The canonical $O_k$-linear homomorphism $M \to M^{\vee \vee}$ is associated to the canonical pairing $M \times M^{\vee} \to k$, which is compatible with the canonical pairing $M \times \rC_0(I,k) = \rC_0(I,k)^{\vee} \times \rC_0(I,k) \to k$ and the canonical $k$-linear homomorphism $\rC_0(I,k) \to \rC_0(I,k)^{\vee \vee}$, and hence is an isomorphism by (1).
\end{proof}

\begin{crl}
\label{adjoint property}
Suppose that $k$ is non-spherically complete.
\bi
\item[(1)] For any free Banach $k$-vector spaces $V_0$ and $V_1$, the map
\be
\Hom_{\leq 1}(V_0,V_1) \to \Hom_{O_k}(V_1^{\vee},V_0^{\vee})
\ee
assigning to each contracting $k$-linear homomorphism $\phi \colon V_0 \to V_1$ the $O_k$-linear homomorphism
\be
\phi^{\vee} \colon V_1^{\vee} & \to & V_0^{\vee} \\
m & \mapsto & m \circ \phi
\ee
is an $O_k$-linear isomorphism.
\item[(2)] For any cofree $O_k$-modules $M_0$ and $M_1$, the map
\be
\Hom_{O_k}(M_0,M_1) \to \Hom_{O_k}(M_1^{\vee},M_0^{\vee})
\ee
assigning to each $O_k$-linear homomorphism $\psi \colon M_0 \to M_1$ the $O_k$-linear homomorphism
\be
\psi^{\vee} \colon M_1^{\vee} & \to & M_0^{\vee} \\
v & \mapsto & v \circ \phi
\ee
is an $O_k$-linear isomorphism onto $\Hom_{\leq 1}(M_1^{\vee},M_0^{\vee})$.
\ei
\end{crl}

\begin{proof}
Injectivity (resp.\ Bijectivity) of a map $\iota$ is equivalent to injectivity (resp.\ bijectivity) of $\iota' \circ \iota$ for some map (resp.\ some injective map) $\iota'$ for which $\iota' \circ \iota$ makes sense. Therefore, it suffices to show the following:
\bi
\item[(1)] Bijectivity of the composite of the map in (1) followed by the map in (2) applied to $(V_1^{\vee},V_0^{\vee})$.
\item[(2)] Bijectivity of the composite of the map in (2) followed by the map in (1) applied to $(M_1^{\vee},M_0^{\vee})$.
\ei
(1) The composite of the given map followed by the map in (2) applied to $(V_1^{\vee},V_0^{\vee})$ is the map
\be
\Hom_{\leq 1}(V_0,V_1) \to \Hom_{O_k}(V_0^{\vee \vee},V_1^{\vee \vee})
\ee
compatible with the isomorphisms $V_0 \to V_0^{\vee \vee}$ and $V_1 \to V_1^{\vee \vee}$, and hence is an $O_k$-linear isomorphism onto $\Hom_{\leq 1}(V_1^{\vee \vee},V_0^{\vee \vee})$ by Theorem \ref{Schikhof} (1).

\vs
(2) Let $\psi$ be an $O_k$-linear homomorphism $M_0 \to M_1$. For any $v \in M_1^{\vee}$, we have
\be
\n{v \circ \psi} = \sup_{m \in M_0} \v{(v \circ \psi)(m)} = \sup_{m \in M_0} \v{v(\psi(m))} \leq \sup_{m \in M_1} \v{v(m)} = \n{v}.
\ee
This implies $\psi^{\vee} \in \Hom_{\leq 1}(M_1^{\vee},M_0^{\vee})$. The composite of the given map followed by the map in (1) applied to $(M_1^{\vee},M_0^{\vee})$ is the map
\be
\Hom_{O_k}(M_0,M_1) \to \Hom_{O_k}(M_0^{\vee \vee},M_1^{\vee \vee})
\ee
compatible with the isomorphisms $M_0 \to M_0^{\vee \vee}$ and $M_1 \to M_1^{\vee \vee}$, and hence is an $O_k$-linear isomorphism by Theorem \ref{Schikhof} (2).
\end{proof}

\begin{crl}
Suppose that $k$ is non-spherically complete. Then the correspondences $\vee$ define a $\Mod(O_k)$-enriched contravariant equivalence between the $\Mod(O_k)$-enriched category of free Banach $k$-vector spaces and contracting $k$-linear homomorphisms and the $\Mod(O_k)$-enriched category of cofree $O_k$-modules and $O_k$-linear homomorphisms.
\end{crl}

\begin{proof}
The assertion immediately follows from Theorem \ref{Schikhof} and Corollary \ref{adjoint property}.
\end{proof}

\subsection{Automatic Continuity Theorem}
\label{Automatic Continuity Theorem}

Let $M$ be an $O_k$-module. A topology $\tau$ of $M$ is said to be {\it linear} if $M$ forms a topological $O_k$-module with respect to $\tau$ and $\tau$ admits a fundamental system of neighbourhoods of $0 \in M$ consisting of $O_k$-submodules of $M$. For a set $S$ of $O_k$-submodules of $M$, {\it the linear topology of $M$ generated by $S$} is the weakest linear topology of $M$ for which every $L \in S$ is open. We characterise the weak topology of a cofree $O_k$-module.

\begin{prp}
\label{characterisation of weak topology}
Suppose that $k$ is non-spherically complete.
\bi
\item[(1)] For any non-measurable set $I$, the weak topology of $O_k^I$ coincides with the direct product topology.
\item[(2)] The weak topology of $O_k^{\N}$ coincides with the linear topology generated by the set $\set{\set{m \in O_k^{\N}}{\sup_{i \in \N} r^i m(i)} \leq \epsilon}{(r,\epsilon) \in (0,1)^2}$ of $O_k$-submodules.
\ei
\end{prp}

\begin{proof}
(1) Since $k$ is non-spherically complete, the canonical pairing $\rC_0(I,k) \times O_k^I \to k$ induces an isometric $k$-linear isomorphism $\rC_0(I,k) \to (O_k^I)^{\vee}$ by \cite{MN89} \S 12 Corollary 7.20, through which we identify $(O_k^I)^{\vee}$ with $\rC_0(I,k)$. For any $i \in I$, the evaluation $O_k^I \twoheadrightarrow O_k$ at $i$ is given by the evaluation at $\delta_i \in \rC_0(I,k) \cong (O_k^I)^{\vee}$. Therefore, every subset of $O_k^I$ open with respect to the direct product topology is weakly open.

\vs
Since both topologies are linear, it suffices to show that weakly open $O_k$-submodules in a subbase of a fundamental system of neighbourhoods of $0$ are open with respect to the direct product topology. Let $(f,\epsilon) \in (O_k^I)^{\vee} \times \R_{> 0}$, and set $L \coloneqq \set{m \in O_k^I}{\v{f(m)} \leq \epsilon}$. We denote by $f' \in \rC_0(I,k)$ the map
\be
I & \to & k \\
i & \mapsto &
\left\{
\begin{array}{ll}
f(i) & (\v{f(i)} > \epsilon) \\
0 & (\v{f(i)} \leq \epsilon)
\end{array}
\right..
\ee
Then we have
\be
L = \set{m \in O_k^I}{\v{\sum_{i \in I} f'(i) m(i)} \leq \epsilon}
\ee
by $\n{f - f'} \leq \epsilon$. Set $I_0 \coloneqq \set{i \in I}{f'(i) \neq 0}$. We have $\# I_0 = \# \set{i \in I}{\v{f(i)} \geq \epsilon} < \aleph_0$ by $f \in (O_k^I)^{\vee} \cong \rC_0(I,k)$. Therefore, we have
\be
\set{m \in O_K^I}{\forall i \in I_0[\v{f'(i) m(i)} \leq \epsilon]} \subset \set{m \in O_k^I}{\v{\sum_{i \in I} f'(i) m_i} \leq \epsilon} = L,
\ee
and hence $L$ is open with respect to the direct product topology.

\vs
(2) For an $(r,\epsilon) \in (0,1)^2$, we set
\be
L_{r,\epsilon} \coloneqq \set{m \in O_k^{\N}}{\sup_{i \in \N} r^i \v{m(i)} \leq \epsilon}.
\ee
We denote by $\tau$ the linear topology of $O_k^{\N}$ generated by the set $\set{L_{r,\epsilon}}{(r,\epsilon) \in (0,1)^2}$ of $O_k$-submodules. Since both topologies are linear, it suffices to compare $O_k$-submodules in subbases of fundamental systems of neighbourhoods of $0$.

\vs
First, let $(r,\epsilon) \in (0,1)^2$. By $0 < r < 1$, there exists an $i_0 \in \N$ such that $0 < r^{i_0} \leq \epsilon$. We have
\be
L_{r,\epsilon} = \set{m \in O_k^{\N}}{\forall i \in \N_{< i_0}[\v{m(i)} \leq r^{-i} \epsilon},
\ee
and hence $L$ is weakly open by (1).

\vs
Next, let $(i_0,\epsilon) \in \N \times (0,1)$ and set
\be
L \coloneqq \set{m \in O_k^I}{\forall i \in \N_{< i_0}[\v{m(i)} \leq \epsilon]}.
\ee
We note that open $O_k$-submodules of this type form a fundamental system of neighbourhoods of $0$ for the weak topology by (1). We have
\be
L_{\epsilon,\epsilon^{i_0}} = \set{m \in O_k^{\N}}{\sup_{i \in \N} \epsilon^i \v{m(i)} \leq \epsilon^{i_0}} \subset \set{m \in O_k^I}{\forall i \in \N_{< i_0}[\v{m(i)} \leq \epsilon]} = L,
\ee
and hence $L$ is open with respect to the $\tau$.
\end{proof}

Proposition \ref{characterisation of weak topology} (1) implies that the direct product topology of a cofree $O_k$-module does not depend on the choice of the topological basis, i.e.\ every $O_k$-linear isomorphism $O_k^{I_0} \to O_k^{I_1}$ is a homeomorphism with respect to the direct product topologies for any non-measurable sets $I_0$ and $I_1$. We have a more general result referred to as {\it automatic continuity theorem} in various contexts (cf.\ \cite{Woo93} and \cite{Dal00} for Archimedean cases, and \cite{Mih14} Theorem 4.6 for a non-Archimedean case).

\begin{thm}[Automatic continuity theorem for cofree $O_k$-modules]
\label{automatic continuity}
Suppose that $k$ is non-spherically complete. Then for any $O_k$-module $M_0$ and any cofree $O_k$-module $M_1$, every $O_k$-linear homomorphism $M_0 \to M_1$ is weakly continuous.
\end{thm}

\begin{proof}
By Proposition \ref{characterisation of weak topology} (1), it suffices to show that for any $O_k$-module $M$, every $O_k$-linear homomorphism $\psi \colon M \to O_k$ is continuous with respect to the weak topology of $M$ and the valuation topology of $O_k$. However, the continuity of $\psi$ trivially follows from the definition of the weak topology.
\end{proof}

\begin{crl}
Suppose that $k$ is non-spherically complete. Then for any cofree $O_k$-module $M$, every $O_k$-linear automorphism $M \to M$ is a weak homeomorphism.
\end{crl}

\begin{proof}
The assertion is a direct consequence of Theorem \ref{automatic continuity}.
\end{proof}

We observe the relation between the weak topology and the Banach localisation.

\begin{prp}
\label{cofree scalar extension}
Let $M$ be a cofree $O_k$-module.
\bi
\item[(1)] If $k$ is non-spherically complete, then the weak topology of $kM$ is Hausdorff.
\item[(2)] The canonical embedding $M \hookrightarrow kM$ is a weak homeomorphism onto the weakly closed image.
\ei
\end{prp}

\begin{proof}
(1) Since $kM$ admits an isometric $k$-linear isomorphism to $\ell^{\infty}(I,k)$ for some non-measurable set $I$, the assertion follows from the injectivity of the canonical $k$-linear homomorphism $kM \to (kM)^{**}$ (cf.\ \cite{MN89} \S 12 Lemma 7.19 (1)).

\vs
(2) By the first assertion of Proposition \ref{operator norm} applied to $kM$, the restriction map $(kM)^* \to M^{\vee}$ is bijective. Since $(kM)^{\vee}$ generates $(kM)^*$, continuity with respect to the evaluations by all elements of $M^{\vee}$ is equivalent to continuity with respect to the evaluations by all elements of $(kM)^{\vee}$. Take an $O_k$-linear isomorphism $\psi \colon M \to O_k^I$. We denote by $k \psi \colon kM \to k(O_k^I)$ the scalar extension of $\psi$. For each $i \in I$, we denote by $\ev_i$ the evaluation $\ell^{\infty}(I,k) \to k$ at $i$. Identifying $k(O_k^I)$ with $\ell^{\infty}(I,k)$ by the canonical isometric isomorphism, we obtain
\be
M & = & \set{m \in kM}{\n{m}_M \leq 1} = \set{m \in kM}{\n{(k \psi)(m)} \leq 1} \\
& = & \set{m \in kM}{\forall i \in I[(k \psi)(m)(i) \in O_k]} = \set{m \in kM}{\forall i \in I[(\ev_i \circ k \psi)(m) \in O_k]},
\ee
and hence $M$ is weakly closed in $kM$.
\end{proof}

A topological $O_k$-module $M$ is said to be {\it linear} if its topology is linear, is said to be a {\it locally convex $k$-vector space} if it is linear and its underlying $O_k$-module forms a $k$-vector space, and is said to be {\it complete} if every Cauchy net in $M$ converges to a unique element. In particular, the completeness implies the Hausdorffness in our context. We introduce a weak variant of the completeness and an algebraic variant of a compactoid.

\begin{dfn}
\label{compactoid}
An $O_k$-module $M$ is said to be {\it weakly complete} if it is complete with respect to the weak topology, is said to be an {\it algebraic compactoid over $O_k$} if for any weakly open $O_k$-submodule $L \subset M$, there exists an $F \in \cP_{< \aleph_0}(M)$ such that $M = L + \sum_{m \in F} O_k m$, and is said to be {\it embeddable} if there exists an $O_k$-linear homomorphism $\iota \colon M \to V$ which is a homeomorphism onto the image with respect to the weak topology of $M$ for some Hausdorff locally convex $k$-vector space $V$.
\end{dfn}

By the definition, an $O_k$-module is an embeddable algebraic compactoid over $O_k$ if and only if it is a compactoid in the sense \cite{Sch95} \S 4.2 with respect to the weak topology. We give an algebraic variant of \cite{Sch95} Proposition 3.1.

\begin{prp}
\label{cofreeness implies compactoid}
Suppose that $k$ is non-spherically complete. Every cofree $O_k$-module is a weakly complete embeddable algebraic compactoid over $O_k$.
\end{prp}

\begin{proof}
It suffices only to consider $O_k^I$ for a non-measurable set $I$. By the completeness of $O_k$ and Proposition \ref{characterisation of weak topology} (1), $O_k^I$ is weakly complete. We denote by $\iota$ the inclusion $O_k^I \hookrightarrow k(O_k^I)$. By Proposition \ref{cofree scalar extension} (1) and (2), $k(O_k^I)$ is a Hausdorff locally convex $k$-vector space with respect to the weak topology, and $\iota$ is a weak homeomorphism onto the image. Therefore, $O_k^I$ is embeddable.

\vs
Let $L \subset O_k^I$ be a weakly open $O_k$-submodule. By Proposition \ref{characterisation of weak topology} (1), there exists an $(I_0,\epsilon) \in \cP_{< \aleph_0}(I) \times \R_{> 0}$ such that $\set{m \in O_k^I}{\forall i \in I_0[\v{m(i)} \leq \epsilon]} \subset L$. Set $F \coloneqq \set{\delta_i}{i \in I_0} \subset O_k^I$. We have $O_k^I = L + \sum_{m \in F} O_k m$.
\end{proof}

\section{Schneider--Teitelbaum Duality}
\label{Schneider--Teitelbaum Duality}

We regard the duality in \S \ref{Schikhof Duality} as a duality of representations of the trivial group, and extend it to a duality of representations of a profinite group.

\subsection{Categorical Equivalence}
\label{Categorical Equivalence 2}

For Banach $k$-vector spaces $V_0$ and $V_1$, we denote by $\Hom_{\leq 1}(V_0,V_1)_{\r{s}}$ the $O_k$-module $\Hom_{\leq 1}(V_0,V_1)$ equipped with the topology of pointwise convergence. For $O_k$-modules $M_0$ and $M_1$, we denote by $\Hom_{O_k}(M_0,M_1)_{\r{b}}$ the $O_k$-module $\Hom_{O_k}(M_0,M_1)$ equipped with the topology of uniform convergence with respect to the weak topology on $M_0$ and the given topology of $M_1$. We give a non-spherically complete analogue of \cite{ST02} Lemma 1.6 and \cite{Mih21-2} Proposition 2.3.

\begin{prp}
\label{hom topology}
Suppose that $k$ is non-spherically complete. Then for any free Banach $k$-vector spaces $V_0$ and $V_1$, the $O_k$-linear isomorphism in Corollary \ref{adjoint property} (1) is a homeomorphism $\Hom_{\leq 1}(V_0,V_1)_{\r{s}} \to \Hom_{O_k}(V_1^{\vee},V_0^{\vee})_{\r{b}}$. 
\end{prp}

\begin{proof}
We denote by $\iota$ the given map. The linear topology on $\Hom_{\leq 1}(V_0,V_1)_{\r{s}}$ is generated by the set
\be
\set{\set{\phi \in \Hom_{\leq 1}(V_0,V_1)}{\n{\phi(v)} \leq \epsilon}}{(v,\epsilon) \in V_0 \times \R_{> 0}}
\ee
of $O_k$-submodules, and the linear topology on $\Hom_{O_k}(V_1^{\vee},V_0^{\vee})_{\r{b}}$ is generated by the set
\be
\set{\set{\psi \in \Hom_{O_k}(V_1^{\vee},V_0^{\vee})}{\sup_{m \in V_1^{\vee}} \v{v(\psi(m))} \leq \epsilon}}{(v,\epsilon) \in V_0^{\vee \vee} \times \R_{> 0}}
\ee
of $O_k$-submodules, which coincides with the set
\be
\set{\set{\psi \in \Hom_{O_k}(V_1^{\vee},V_0^{\vee})}{\sup_{m \in V_1^{\vee}} \v{\psi(m)(v)} \leq \epsilon}}{(v,\epsilon) \in V_0 \times \R_{> 0}}
\ee
by Theorem \ref{Schikhof} (2).

\vs
Let $(v,\epsilon) \in V_0 \times \R$. Set
\be
U & \coloneqq & \set{\phi \in \Hom_{\leq 1}(V_0,V_1)}{\n{\phi(v)} \leq \epsilon} \\
U' & \coloneqq & \set{\psi \in \Hom_{O_k}(V_1^{\vee},V_0^{\vee})}{\sup_{m \in V_1^{\vee}} \v{\psi(m)(v)} \leq \epsilon}.
\ee
By the argument above, it suffices to show $\iota(U) = U'$.

\vs
First, let $\phi \in U$. For any $m \in V_1^{\vee}$, we have
\be
\v{\iota(\phi)(m)(v)} = \v{(m \circ \phi)(v)} = \v{m(\phi(v))} \leq \n{\phi(v)} \leq \epsilon.
\ee
This implies $\iota(\phi) \in U'$.

\vs
Next, let $\psi \in U'$. We have
\be
\n{\iota^{-1}(\psi)(v)} = \sup_{m \in V_1^{\vee}} \v{m(\iota^{-1}(\psi)(v))} = \sup_{m \in V_1^{\vee}} \v{\psi(m)(v)} \leq \epsilon
\ee
by Theorem \ref{Schikhof} (1). This implies $\iota^{-1}(\psi) \in U$.
\end{proof}

We recall a partial generalisation of Banach--Steinhaus theorem (cf. \cite{Sch02} Corollary 6.16) in \cite{Mih21-2} Proposition 3.1.

\begin{prp}
\label{Banach Currying}
Let $X$ be a topological space, and $V_0$ and $V_1$ Banach $k$-vector spaces. For any map $\rho \colon X \to \Hom_{\leq 1}(V_0,V_1)_{\r{s}}$, the following are equivalent:
\bi
\item[(1)] The map $\rho$ is continuous.
\item[(2)] The map
\be
X \times V_0 & \to & V_1 \\
(x,v) & \mapsto & \rho(x)(v)
\ee
is continuous.
\ei
\end{prp}

\begin{proof}
The assertion is identical to that of \cite{Mih21-2} Proposition 3.1 except that $k$ was assumed to be a local field throughout \cite{Mih21-2}. Since the proof of \cite{Mih21-2} Proposition 3.1 did not use the assumption, it works for a general complete valuation field (even possibly with trivial valuation).
\end{proof}

\begin{prp}
\label{dual Currying}
Let $X$ be a topological space, $M_0$ an algebraic compactoid over $O_k$ (cf.\ Definition \ref{compactoid}), and $M_1$ an $O_k$-module. For any map $\rho \colon X \to \Hom_{O_k}(M_0,M_1)_{\r{b}}$, the following are equivalent:
\bi
\item[(1)] The map $\rho$ is continuous.
\item[(2)] The map
\be
X \times M_0 & \to & M_1 \\
(x,m) & \mapsto & \rho(x)(m)
\ee
is continuous with respect to the weak topologies of $M_0$ and $M_1$.
\ei
\end{prp}

\begin{proof}
Although the implication from (1) to (2) follows from the first half part of the proof of \cite{Mih21-2} Proposition 3.7 on compact uniform spaces because the assumption of the compactness was used only in the last half part, we directly show it for the reader's convenience. We denote by $\pi$ the map in (2).

\vs
First, assume that $\rho$ is continuous. Let $(x,m,v,\epsilon) \in X \times M_0 \times M_1^{\vee} \times \R_{> 0}$ and set
\be
U \coloneqq \set{m' \in M_1}{\v{v(m' - \pi(x,m))} \leq \epsilon}.
\ee
We show that $\pi^{-1}(U)$ is a neighbourhood of $(x,m)$. Set
\be
\cU \coloneqq \set{\psi \in \Hom_{O_k}(M_0,M_1)}{\sup_{m' \in M_0} \v{v((\psi - \rho(x))(m'))} \leq \epsilon}.
\ee
By the continuity of $\rho$, $\rho^{-1}(\cU)$ is a neighbourhood of $x$. For any $(x',m') \in \rho^{-1}(\cU) \times M_0$ with $\v{\rho(x)^{\vee}(v)(m' - m)} \leq \epsilon$, we have
\be
\v{v(\pi(x',m') - \pi(x,m))} & = & \v{v(\rho(x')(m') - \rho(x)(m))} \\
& \leq & \max \ens{\v{v(\rho(x')(m') - \rho(x)(m'))}, \v{v(\rho(x)(m') - \rho(x)(m))}} \\
& = & \max \ens{\v{v((\rho(x') - \rho(x))(m'))}, \v{v(\rho(x)(m' - m))}} \\
& \leq & \max \ens{\v{v((\rho(x') - \rho(x))(m'))}, \v{\rho(x)^{\vee}(v)(m' - m)}} \leq \epsilon,
\ee
and hence $(x',m') \in \pi^{-1}(U)$. Therefore, $\pi^{-1}(U)$ is a neighbourhood of $(x,m)$ with respect to the weak topology of $M_0$.

\vs
Next, assume that $\pi$ is continuous with respect to the weak topologies of $M_0$ and $M_1$. Let $(x,v,\epsilon) \in X \times M_1^{\vee} \times \R_{> 0}$ and set
\be
\cU \coloneqq \set{\psi \in \Hom_{O_k}(M_0,M_1)}{\sup_{m \in M_0} \v{v((\psi - \rho(x))(m))} \leq \epsilon}.
\ee
We show that $\rho^{-1}(\cU)$ is a neighbourhood of $x$. Set $L' \coloneqq \set{m \in M_1}{\v{v(m)} \leq \epsilon}$. Since $\pi$ is continuous at $(x,0)$, there exists a pair $(U,L)$ of a neighbourhood $U$ of $x$ and a weakly open $O_k$-submodule $L \subset M_0$ such that $\pi(U \times L) \subset L'$. Since $M_0$ is an algebraic compactoid over $O_k$, there exists an $F \in \cP_{< \aleph_0}(M_0)$ such that $M_0 = L + \sum_{m \in F} O_k m$. For each $m \in F$, take a neighbourhood $U_m$ of $x$ such that $\pi(U_m \times \ens{m}) \subset \pi(x,m) + L'$, which exists by the continuity of $\pi$ at $(x,m)$. Set $U' \coloneqq \bigcap (\ens{U} \cup \set{U_m}{m \in F})$.

\vs
Let $x' \in U'$. For any $m \in M_0$, there exists an $(m_0,(c_{m'})_{m' \in F}) \in L \times O_k^F$ such that $m = m_0 + \sum_{m' \in F} c_{m'} m'$ by $M_0 = L + \sum_{m' \in F} O_k m'$, and hence we have
\be
(\rho(x') - \rho(x))(m) & = & \pi(x',m_0) - \pi(x,m_0) - \sum_{m' \in F} c_{m'} (\pi(x',m') - \pi(x,m')) \\
& \in & \pi(U \times L) + \pi(U \times L) + \sum_{m' \in F} O_k (\pi(U_{m'} \times \ens{m'}) - \pi(x,m')) \subset L',
\ee
i.e.\ $\v{v((\rho(x') - \rho(x))(m))} \leq \epsilon$. This implies $x' \in \rho^{-1}(\cU)$, Therefore, $\rho^{-1}(\cU)$ is a neighbourhood of $x$.
\end{proof}

\begin{prp}
\label{completeness of b-topoogy}
For any cofree $O_k$-modules $M_0$ and $M_1$, $\Hom_{O_k}(M_0,M_1)_{\r{b}}$ is a complete linear topological $O_k$-module.
\end{prp}

\begin{proof}
The assertion immediately follows from the weak completeness of $M_1$ (cf.\ Proposition \ref{cofreeness implies compactoid}).
\end{proof}

Let $X$ be a compact topological space. We denote by $\rC(X,k)$ the Banach $k$-vector space of continuous functions $X \to k$ equipped it with the supremum norm
\be
\rC(X,k) & \to & \R_{\geq 0} \\
f & \mapsto & \sup_{x \in X} \v{f(x)}.
\ee
For each clopen subset $U \subset X$, we abuse the convention $1_U$ without specifying $X$ for the characteristic function of $U$ regarded as an element of $\rC(X,k)$. The Banach $k$-vector space $\rC(X,k)^*$ is identified with the space of bounded $k$-valued measures on $X$ through the evaluations at $1_U$ for all clopen subsets $U \subset X$, and hence its closed unit ball $\rC(X,k)^{\vee}$ is identified with the space of $O_k$-valued measures on $X$.

\vs
For an $O_k$-module $M_0$ and a topological $O_k$-module $M_1$, we denote by $\Hom_{O_k}^{\r{w}}(M_0,M_1)$ the $O_k$-module of $O_k$-linear homomorphisms $M_0 \to M_1$ continuous with respect to the weak topology of $M_0$ and the given topology of $M_1$. We have the following universality of the measure space $\rC(X,k)^{\vee}$:

\begin{prp}
\label{universality of measure space}
Suppose that $k$ is non-spherically complete. Let $X$ be a compact topological space such that the set of clopen subsets of $X$ is non-measurable.
\bi
\item[(1)] The Banach $k$-vector space $\rC(X,k)$ is free, and the $O_k$-module $\rC(X,k)^{\vee}$ is cofree.
\item[(2)] The map $\ev \colon X \to \rC(X,k)^{\vee}$ assigning to each $x \in X$ the evaluation $\rC(X,k) \to k$ at $x$ is continuous with respect to the weak topology of $\rC(X,k)^{\vee}$, and the image of $\ev$ generates a weakly dense $O_k$-submodule of $\rC(X,k)^{\vee}$.
\item[(3)] For any complete linear topological $O_k$-module $M$, the map
\be
\Hom_{O_k}^{\r{w}}(\rC(X,k)^{\vee},M) & \to & \rC(X,M) \\
\psi & \mapsto & \psi \circ \ev
\ee
is bijective.
\ei
\end{prp}

\begin{proof}
(1) The first assertion follows from \cite{MN89} \S 12 Theorem 7.21, and the second assertion follows from Theorem \ref{Schikhof} (1).

\vs
(2) By Theorem \ref{Schikhof} (1), the canonical $k$-linear homomorphism $\rC(X,k) \to \rC(X,k)^{\vee \vee}$ is bijective. This implies the first assertion by the definition of the weak topology. By Theorem \ref{Schikhof} (1), it suffices to show that for any $(\mu,F,\epsilon) \in \rC(X,k)^{\vee} \times \cP_{< \aleph_0}(\rC(X,k)) \times \R_{> 0}$, there exists a $c \in O_k^{\oplus X}$ such that $\sup_{f \in F} \v{(\mu - \sum_{x \in X} c(x) \ev(x))(f)} \leq \epsilon$. For any $f \in \rC(X,k)$, since $X$ is compact, the coset decomposition of $X$ associated to $f$ modulo $\set{c \in k}{\v{c} \leq \epsilon}$ gives a finite clopen covering of $X$. Therefore, since $F$ is finite, there exists a finite clopen covering $\cU$ of $X$ such that for any $U \in \cU$, $\sup_{f \in F} \sup_{(x,x') \in U^2} \v{f(x) - f(x')} \leq \epsilon$. Take a section $s \colon \cU \hookrightarrow X$ of the canonical projection $p \colon X \twoheadrightarrow \cU$. We denote by $c \in O_k^{\oplus X}$ the map
\be
X & \to & O_k \\
x & \mapsto & 
\left\{
\begin{array}{ll}
\mu(1_{p(x)}) & (x \in \im(s)) \\
0 & (x \notin \im(s))
\end{array}
\right..
\ee
For any $f \in F$, we have
\be
\n{f - \sum_{x \in \im(s)} f(x) 1_{p(x)}} \leq \epsilon
\ee
by the choice of $\cU$, and hence
\be
& & \v{\left( \mu - \sum_{x \in X} c(x) \ev(x) \right)(f)} = \v{\left( \mu - \sum_{x \in \im(s)} \mu(1_{p(x)}) \ev(x) \right)(f)} = \v{\mu(f) - \sum_{x \in \im(s)} f(x) \mu(1_{p(x)})} \\
& = & \v{\mu \left( f - \sum_{x \in \im(s)} f(x) 1_{p(x)} \right)} \leq \n{f - \sum_{x \in \im(s)} f(x) 1_{p(x)}} \leq \epsilon.
\ee
This implies $\sup_{f \in F} \v{(\mu - \sum_{x \in X} c(x) \ev(x))(f)} \leq \epsilon$.

\vs
(3) We denote by $\ev^*$ the given map, and by $\cO_M$ the set of open $O_k$-submodules of $M$ directed by inclusion. We construct the inverse of $\ev^*$. Let $(\rho,\mu) \in \rC(X,M) \times \rC(X,k)^{\vee}$. For any $L \in \cO_M$, the finite sum $\sum_{x \in X_0} \mu(1_{\rho^{-1}(\rho(x) + L)}) \rho(x) + L \in M/L$ does not depend on the choice of a complete representative $X_0$ of the coset decomposition $X \twoheadrightarrow \set{\rho^{-1}(\rho(x) + L)}{x \in X}$ associated to $\rho$ modulo $L$, and we denote it by $\int \rho d \mu/L$. We have $(\int \rho d \mu/L)_{L \in \cO_M} \in \varprojlim_{L \in \cO_M} M/L$, and hence there uniquely exists an $m \in M$ such that $m + L = \int \rho d \mu/L$ for any $L \in \cO_M$ by the completeness of $M$. We denote such an $m$ by $\int \rho d \mu$.

\vs
We show that the map
\be
\int \rho \colon \rC(X,k)^{\vee} & \to & M \\
\mu & \mapsto & \int \rho d \mu.
\ee
is an $O_k$-linear homomorphism continuous with respect to the weak topology of $\rC(X,k)^{\vee}$. The $O_k$-linearity follows from the definition. Since the weak topology of $\rC(X,k)$ and the topology on $M$ are linear, it suffices to show that for any $L \in \cO_M$, $(\int \rho)^{-1}(L)$ is a neighbourhood of $0 \in \rC(X,k)^{\vee}$. Take a complete representative $X_0$ of the coset decomposition $X \twoheadrightarrow \set{\rho^{-1}(\rho(x) + L)}{x \in X}$. For each $x \in X_0$, since $L$ is a neighbourhood of $0 = 0 \rho(x)$, there exists an $\epsilon_x \in \R_{> 0}$ such that $c \rho(x) \in L$ for any $c \in O_k$ with $\v{c} \leq \epsilon_x$. For any $\mu \in \rC(X,k)^{\vee}$ with $\v{\mu(1_{\rho^{-1}(\rho(x) + L)})} \leq \epsilon_x$ for any $x \in X_0$, we have
\be
\int \rho d \mu + L = \int \rho d \mu / L = \sum_{x \in X_0} \mu(1_{\rho^{-1}(\rho(x) + L)}) \rho(x) + L = 0 + L,
\ee
and hence $\mu \in (\int \rho)^{-1}(L)$. Therefore, $(\int \rho)^{-1}(L)$ is a neighbourhood of $0$.

\vs
We show that the map
\be
\int \colon \rC(X,M) & \to & \Hom_{O_k}^{\r{w}}(\rC(X,k)^{\vee},M) \\
\rho & \mapsto & \int \rho
\ee
is the inverse of $\ev^*$. First, let $\rho \in \rC(X,M)$. We show $(\ev^* \circ \int)(\rho) = \rho$. Let $x \in X$. For any $L \in \cO_M$, we have
\be
\int \rho d (\ev(x)) + L = \int \rho d (\ev(x)) / L = \ev(x)(1_{\rho^{-1}(\rho(x) + L)}) \rho(x) + L = \rho(x) + L
\ee
by the definition of $\int \rho d (\ev(x))/L$ using the independence of the choice of a complete representative of the coset decomposition associated to $\rho$ modulo $L$. By the Hausdorffness of $M$, we obtain $\int \rho d(\ev(x)) = \rho(x)$, and hence
\be
\left( \ev^* \circ \int \right)(\rho)(x) = \int \rho d(\ev(x)) = \rho(x).
\ee
This implies $(\ev^* \circ \int)(\rho) = \rho$.

\vs
Next, let $\psi \in \Hom_{O_k}^{\r{w}}(\rC(X,k)^{\vee},M)$. We show $(\int \circ \ev^*)(\psi) = \psi$. Let $x \in X$. For any $L \in \cO_M$, we have
\be
\left( \int \circ \ev^* \right)(\psi)(\ev(x)) + L & = & \int (\psi \circ \ev) d (\ev(x)) + L = (\psi \circ \ev)(x) + L = \psi(\ev(x)) + L
\ee
by the argument above applied to $\rho = (\psi \circ \ev)$. By the Hausdorffness of $M$, we obtain $(\int \circ \ev^*)(\psi)(\ev(x)) = \psi(\ev(x))$. This implies that the restrictions of $(\int \circ \ev^*)(\psi)$ and $\psi$ to the $O_k$-submodule of $\rC(X,k)^{\vee}$ generated by the image of $\ev$ coincide with each other. By the second assertion of (2), the continuity of $(\int \circ \ev^*)(\psi)$ and $\psi$ with respect to the weak topology of $\rC(X,k)^{\vee}$, and the Hausdorffness of $M$, we obtain $(\int \circ \ev^*)(\psi) = \psi$.
\end{proof}

Let $G$ be a profinite group, and assume that the set $\cO_G$ of open normal subgroups of $G$ directed by inclusion is non-measurable. By Proposition \ref{universality of measure space} (1), $\rC(G,k)$ is free and $\rC(G,k)^{\vee}$ is cofree. We denote by $O_k[[G]]$ the Iwasawa algebra of $G$ over $O_k$, i.e.\ $\varinjlim_{H \in \cO_G} O_k[G/H]$.

\begin{prp}
\label{weak topology vs inverse limit topology}
Suppose that $k$ is non-spherically complete.
\bi
\item[(1)] The map
\be
\rC(G,k)^{\vee} & \to & O_k[[G]] \\
\mu & \mapsto & \left( \sum_{\ol{g} \in G/H} \mu(1_{\ol{g}}) \right)_{H \in \cO_G}
\ee
is an $O_k$-linear isomorphism.
\item[(2)] The weak topology of $O_k[[G]]$ coincides with the inverse limit topology.
\ei
\end{prp}

\begin{proof}
(1) For an $H \in \cO_G$, we denote by $p_H$ the canonical projection $O_k[[G]] \twoheadrightarrow O_k[G/H]$. The assertion follows from the fact that every clopen covering of $G$ is refined by the coset decomposition associated to $p_H$ for some $H \in \cO_G$.

\vs
(2) Since both topologies are linear, it suffices to compare $O_k$-submodules in subbases of fundamental systems of neighbourhoods of $0$.

\vs
First, let $(H,\epsilon) \in \cO_G \times \R_{> 0}$ and set
\be
L \coloneqq p_H^{-1} \left( \set{m \in O_k[G/H]}{\sup_{\ol{g} \in G/H} \v{m(\ol{g})} \leq \epsilon} \right).
\ee
For each $\ol{g} \in G/H$, we denote by $\ev_{\ol{g}}$ the evaluation $O_k[G/H] \to O_k$ at $\ol{g}$. We have
\be
L = \bigcap_{\ol{g} \in G/H} (\ev_{\ol{g}} \circ p_H)^{-1}(\set{c \in k}{\v{c} \leq \epsilon}),
\ee
and hence $L$ is weakly open.

\vs
Next, let $(\psi,\epsilon) \in O_k[[G]]^{\vee} \times \R_{> 0}$, and set
\be
L \coloneqq \set{\mu \in O_k[[G]]}{\v{\psi(\mu)} \leq \epsilon}.
\ee
By Theorem \ref{Schikhof} (1), there uniquely exists an $f \in \rC(G,k)$ such that $\n{f} = \n{\psi}$ and the image of $f$ by the canonical homomorphism $\rC(G,k) \to \rC(G,k)^{\vee \vee}$ coincides with the image of $\psi$ by the dual $O_k[[G]]^{\vee} \to \rC(G,k)^{\vee \vee}$ of the $O_k$-linear homomorphism in the first assertion. Set $\mathfrak{a} \coloneqq \set{c \in O_k}{\v{c} \leq \epsilon}$. Since $f$ is continuous and $G$ is profinite, there exists an $H \in \cO_G$ such that $f$ modulo $\mathfrak{a}$ factors through the canonical projection $G \twoheadrightarrow G/H$. We denote by $\ol{f}$ the resulting map $G/H \to O_k/\mathfrak{a}$. For any $\mu \in O_k[[G]]$, we have
\be
\psi(\mu) + \mathfrak{a} = \sum_{\ol{g} \in G/H} p_H(\mu)(\ol{g}) \ol{f}(\ol{g}).
\ee
Therefore, the inverse image of $\mathfrak{a} O_k[G/H]$ by $p_H$ is contained in $L$. Therefore, $L$ is open with respect to the inverse limit topology.
\end{proof}

Through the isomorphism in Proposition \ref{weak topology vs inverse limit topology} (1), we identify $\rC(G,k)^{\vee}$ with the underlying $O_k$-module of the $O_k$-algebra $O_k[[G]]$. We are ready to formulate Schneider--Teitelbaum duality for a Banach representation over a non-spherically complete field.

\begin{dfn}
A {\it unitary Banach $k$-linear representation of $G$} is a pair $(V,\rho)$ of a free Banach $k$-vector space $V$ and a continuous map $\rho \colon G \times V \to V$ for which the underlying $k$-vector space of $V$ forms a representation of the underlying group of $G$ and the map
\be
V & \to & V \\
v & \mapsto & \rho(g,v)
\ee
is an isometry for any $g \in G$. A left $O_k[[G]]$-module $M$ is said to be {\it cofree} if $M$ is cofree as an $O_k$-module.
\end{dfn}

Let $(V,\rho)$ be a unitary Banach $k$-linear representation of $G$. By Proposition \ref{Banach Currying}, the continuous map
\be
\rho \colon G \times V \to V
\ee
corresponds to a continuous map
\be
G \to \Hom_{\leq 1}(V,V)_{\r{s}},
\ee
which is a monoid homomorphism with respect to the composition of $\Hom_{\leq 1}(V,V)_{\r{s}}$. By Proposition \ref{hom topology}, it corresponds to a continuous map
\be
G \to \Hom_{O_k}(V^{\vee},V^{\vee})_{\r{b}},
\ee
which is an anti-monoid homomorphism with respect to the composition of $\Hom_{O_k}(V^{\vee},V^{\vee})_{\r{b}}$. Composing it with the homeomorphic anti-group isomorphism
\be
G & \to & G \\
g & \mapsto & g^{-1},
\ee
we obtain a continuous monoid homomorphism. By Proposition \ref{completeness of b-topoogy} and Proposition \ref{universality of measure space}, it corresponds to a continuous $O_k$-linear homomorphism
\be
O_k[[G]] \to \Hom_{O_k}(V^{\vee},V^{\vee})_{\r{b}}
\ee
with respect to the weak topology on $O_k[[G]]$, which is an $O_k$-algebra homomorphism with respect to the multiplication given by the composition of $\Hom_{O_k}(V^{\vee},V^{\vee})_{\r{b}}$. By Proposition \ref{cofreeness implies compactoid} and Proposition \ref{dual Currying}, it corresponds to a continuous map
\be
O_k[[G]] \times V^{\vee} \to V^{\vee}
\ee
with respect to the weak topologies of $O_k[[G]]$ and $V^{\vee}$, for which $V^{\vee}$ forms a cofree left $O_k[[G]]$-module. We denote by $(V,\rho)^{\vee}$ the resulting cofree left $O_k[[G]]$-module.

\begin{thm}[Schneider--Teitelbaum duality over a non-spherically complete field]
\label{Schneider--Teitelbaum}
Suppose that $k$ is non-spherically complete. The correspondence $(V,\rho) \mapsto (V,\rho)^{\vee}$ defines a $\Mod(O_k)$-enriched contravariant equivalence between the $\Mod(O_k)$-enriched category of unitary Banach $k$-linear representations of $G$ and contracting $G$-equivariant $k$-linear homomorphisms and the $\Mod(O_k)$-enriched category of cofree left $O_k[[G]]$-modules and $O_k[[G]]$-linear homomorphisms.
\end{thm}

In order to prove Theorem \ref{Schneider--Teitelbaum}, we prepare convention and a lemma.

\begin{dfn}
A left $O_k[[G]]$-module $M$ is said to be {\it weakly topological} if the multiplication $O_k[[G]] \times M \to M$ is continuous with respect to the weak topologies of $O_k[[G]]$ and $M$.
\end{dfn}

\begin{lmm}
\label{automatic continuity of multiplication}
Suppose that $k$ is non-spherically complete. Then every cofree left $O_k[[G]]$-module $M$ is weakly topological.
\end{lmm}

\begin{proof}
By Proposition \ref{characterisation of weak topology} (1), the direct product topology of $O_k[[G]] \times M$ with respect to the weak topologies of $O_k[[G]]$ and $M$ coincides with the weak topology. Therefore, the multiplication $O_k[[G]] \times M \to M$ is continuous with respect to the weak topologies of $O_k[[G]]$ and $M$ by Theorem \ref{automatic continuity}.
\end{proof}

\begin{proof}[Proof of Theorem \ref{Schneider--Teitelbaum}]
By the construction, the correspondence $(V,\rho) \to (V,\rho)^{\vee}$ defines a contravariant functor from the $\Mod(O_k)$-enriched category of unitary Banach $k$-linear representations of $G$ and contracting $G$-equivariant $k$-linear homomorphisms and the $\Mod(O_k)$-enriched category of weakly topological cofree left $O_k[[G]]$-modules and $O_k[[G]]$-linear homomorphisms, which is a composite of categorical equivalence. Therefore, the assertion immediately follows from Lemma \ref{automatic continuity of multiplication}.
\end{proof}

\subsection{Weak Irreducibility}
\label{Weak Irreducibility}

As an application of Theorem \ref{Schneider--Teitelbaum}, we give a criterion of a weak variant of irreducibility of a unitary representation of a profinite group.

\begin{dfn}
A subset of a unitary Banach $k$-linear representation $(V,\rho)$ of $G$ is said to be a {\it subrepresentation} if it is a $G$-stable $k$-linear subspace of $(V,\rho)$. A left $O_k[[G]]$-submodule of a left $O_k[[G]]$-module $M$ is said to be {\it generically weakly closed pure} if it is generically weakly closed and pure as an $O_k$-submodule.
\end{dfn}

\begin{dfn}
A unitary Banach $k$-linear representation of $G$ is said to be {\it irreducible} (resp.\ {\it weakly irreducible}) if it admits precisely two closed (resp.\ weakly closed) subrepresentations. A left $O_k[[G]]$-module is said to be {\it generically weakly simple} if it admits precisely two generically weakly closed pure left $O_k[[G]]$-submodules of $M$.
\end{dfn}

By the first sentence of Example \ref{weak closedness vs closedness}, irreducibility of a unitary representation of $G$ implies weak irreducibility. Although both of irreducibility and weak irreducibility are topological notions, the latter one can be interpreted into the purely algebraic notion of weak generic simplicity through the following:

\begin{thm}
\label{weak irreducibility criterion}
Suppose that $k$ is non-spherically complete. For any unitary Banach $k$-linear representation $(V,\rho)$ of $G$, $(V,\rho)$ is weakly irreducible if and only if $(V,\rho)^{\vee}$ is generically weakly simple.
\end{thm}

In order to prove Theorem \ref{weak irreducibility criterion}, we prepare convention and lemmata.

\begin{lmm}
\label{weak integrability of an action}
Suppose that $k$ is non-spherically complete. Let $M$ be a cofree left $O_k[[G]]$-module. For any weakly closed $O_k$-submodule $L \subset M$, $L$ is a left $O_k[[G]]$-submodule if and only if $L$ is $G$-stable.
\end{lmm}

\begin{proof}
The inverse implication is obvious. Assume that $L$ is $G$-stable. We show $\mu m \in L$ for any $(\mu,m) \in O_k[[G]] \times L$. Since $L$ is weakly closed, it suffices to show $\mu m \in L + L'$ for any weakly open $O_k$-submodule $L' \subset M$. For each $H \in \cO_G$, we denote by $p_H$ the canonical projection $O_k[[G]] \twoheadrightarrow O_k[G/H]$.

\vs
By Proposition \ref{dual Currying} and Lemma \ref{automatic continuity of multiplication}, the map $O_k[[G]] \to \Hom_{O_k}(M,M)_{\r{b}}$ associated to the multiplication is continuous with respect to the weak topology of $O_k[[G]]$. Therefore, by Proposition \ref{weak topology vs inverse limit topology} (2), there exists an $(H,\epsilon) \in \cO_G \times \R_{> 0}$ such that $\mu' m' \in L'$ for any $(\mu',m') \in O_k[[G]] \times M$ with $\sup_{\ol{g} \in G/H} \v{p_H(\mu')(\ol{g})} \leq \epsilon$. In particular, we have $\ker(p_H) m' \subset L'$ for any $m' \in M$. Take a complete representative $G_0$ of the canonical projection $G \twoheadrightarrow G/H$. By the $G$-stability of $L$ and
\be
\mu - \sum_{g \in G_0} p_H(\mu)(gH) [g] \in \ker(p_H),
\ee
we obtain
\be
\mu m \in \sum_{g \in G_0} p_H(\mu)(gH) [g] m + \ker(p_H) m \subset L + L'.
\ee
\end{proof}

For a unitary Banach $k$-linear representation $(V,\rho)$ of $G$, we denote by $\WC(V,\rho)$ the set of weakly closed subrepresentations of $(V,\rho)$ directed by inclusion. For a left $O_k[[G]]$-module $M$, we denote by $\WC(M)$ the set of generically weakly closed pure left $O_k[[G]]$-submodules directed by inclusion.

\begin{lmm}
\label{equivariant perp equivalence}
Suppose that $k$ is non-spherically complete. Let $(V,\rho)$ be a unitary Banach $k$-linear representation of $G$. The correspondences $\perp$ define anti-order-preserving maps between $\WC(V,\rho)$ and $\WC((V,\rho)^{\vee})$ which are inverses to each other.
\end{lmm}

\begin{proof}
By the definition of $\vee$ and $\perp$, $\perp$ preserves $G$-stability. Therefore, the assertion follows from Corollary \ref{perp equivalence} (1) and Lemma \ref{weak integrability of an action}.
\end{proof}

\begin{proof}[Proof of Theorem \ref{weak irreducibility criterion}]
Lemma \ref{equivariant perp equivalence} implies $\# \WC(V,\rho) = \# \WC((V,\rho)^{\vee})$, which implies the assertion.
\end{proof}

\section{Example}
\label{Example}

Let $p$ be a prime number. We give two $p$-adic families of infinite dimensional weakly irreducible unitary Banach $\Cp$-linear representations of a $p$-adic Lie group. Both are quite similar to families in preceding studies over finite extensions of $\Qp$, but are novel by the weak irreducibility over $\Cp$. For a $(U,i,j) \in \GL_2(\Zp) \times \ens{0,1}^2$, we denote by $U_{i,j} \in \Zp$ the $(1+i,1+j)$-entry of $U$.

\subsection{Symmetric Product of Complex $p$-adic Weight}
\label{Symmetric Product of Complex p-adic Weight}

First, we construct a $p$-adic family interpolating symmetric products in a way similar to one in \cite{Mih15} \S 3. We denote by $\rho$ the map
\be
\GL_2(\Zp) \times \Cp^2 & \to & \Cp^2 \\
(U,v) & \mapsto &
\left(
\begin{array}{c}
U_{0,0} v(0) + U_{0,1} v(1) \\
U_{1,0} v(0) + U_{1,1} v(1) \\
\end{array}
\right),
\ee
i.e.\ the scalar extension to $\Cp$ of the canonical representation of $\GL_2(\Zp)$. We set $X \coloneqq (1,0) \in \Cp^2$ and $Y \coloneqq (0,1) \in \Cp^2$ in order to identify $\Cp^2$ with the $\Cp$-linear subspace of two-variable polynomials consisting of homogeneous polynomials of degree $1$.

\vs
For an $n \in \N$, we denote by $\Sym^n(\rho)$ the $n$-th symmetric product $\GL_2(\Zp) \times \Sym^n_{\Cp}(\Cp^2) \to \Sym^n_{\Cp}(\Cp^2)$ of $\rho$. By the definition, we have
\be
U X^{n-j} Y^j & = & (U_{0,0} X + U_{1,0} Y)^{n-j} (U_{0,1} X + U_{1,1} Y)^j \\
& = & \sum_{i=0}^{n} \left( \sum_{h = \max \ens{0,i+j-n}}^{\min \ens{i,j}} \binom{n-j}{i-h} \binom{j}{h} U_{0,0}^{n-i-j+h} U_{1,0}^{i-h} U_{0,1}^{j-h} U_{1,1}^{h} \right) X^{n-i} Y^i
\ee
for any $(n,j,U) \in \N \times \N_{\leq n} \times \GL_2(\Zp)$. We interpolate this formula to obtain an infinite family of infinite dimensional unitary Banach $\Cp$-linear representations. We set
\be
B \coloneqq \set{U \in \GL_2(\Zp)}{U_{1,0} \in p \Zp}.
\ee
By $\# \cO_B = \aleph_0$, duality theory in \S \ref{Schneider--Teitelbaum Duality} is applicable to unitary Banach $k$-linear representations of $B$. We denote by $W$ the group of continuous characters $\Zp^{\times} \to \Cp^{\times}$. We employ the additive convention for $W$, and identify the additive group $\Z$ with its image by the natural embedding $\Z \hookrightarrow W$. For each $(n,u) \in W \times \Zp^{\times}$, we abbreviate $n(u)$ to $u^n$. We set $s_p \coloneqq 2$ and $q_p \coloneqq 4$ if $p = 2$, and $s_p \coloneqq p-1$ and $q_p \coloneqq p$ otherwise. For each $n \in W$, we denote by $n^{(p)} \in \Z/s_p \Z$ the exponent corresponding to the restriction of $n$ to $\set{u \in \Zp^{\times}}{u^{s_p} = 1}$, and by $n_p \in O_{\Cp}$ the exponent corresponding to the restriction of $n$ to $1 + q_p \Zp$.

\begin{dfn}
Let $f$ be a function $W \to \Cp$. We say that $f$ is {\it rigid analytic} if there exists an $(f_r)_{r \in \Z/s_p \Z} \in \Cp \ens{T}^{\Z/s_p \Z}$ such that for any $n \in W$, the equality $f(n) = f_{n^{(p)}}(n_p)$ holds. We denote by $A_W$ the $\Cp$-algebra of rigid analytic functions on $W$.
\end{dfn}

We equip $A_W$ with supremum norm. Since the supremum norm of an $f \in A_W$ coincides with the maximum of Gauss norms of $f_r$'s for the corresponding $(f_r)_{r \in \Z/s_p \Z}$, $A_W$ forms a Banach $\Cp$-algebra. For each $(z,i) \in \Cp \times \N$, we set
\be
\binom{z}{i} \coloneqq \frac{1}{i!} \prod_{h=0}^{i-1} (z - h) \in \Cp.
\ee
For each $(n,i) \in W \times \N$, we set
\be
\binom{n}{i} \coloneqq \binom{n_p}{i} \in \Cp.
\ee
For any $i \in \N$, $\binom{n}{i}$ is rigid analytic on $n \in W$ because it is defined by a polynomial on $n_p$. By Legendre's formula, we have
\be
\v{\binom{n}{i}} \ 
\left\{
\begin{array}{ll}
= 1 & (i = 0) \\
< \v{p}^{- \frac{i}{p-1}} & (i > 0)
\end{array}
\right.
\ee
for any $(n,i) \in W \times \N$. In particular, for any $i \in \N$, the norm of the rigid analytic function $\binom{n}{i}$ on $n \in W$ is $1$ if $i = 0$ and is bounded by $\v{p}^{- \frac{i}{p-1}}$ otherwise.

\vs
For each $(n,i,j) \in W \times \N^2$, we define a map $\rho_{n,i,j} \colon B \to \Cp$ as
\be
\rho_{n,i,j}(U) \coloneqq \sum_{h=0}^{\min \ens{i,j}} \binom{n-j}{i-h} \binom{j}{h} U_{0,0}^{n-i-j+h} U_{0,1}^{j-h} U_{1,0}^{i-h} U_{1,1}^{h} \in \Cp.
\ee
For any $(n,i,j,h,U) \in W \times \N^3 \times B$ with $h \leq \min \ens{i,j}$, we have
\be
\v{\binom{n-j}{i-h} U_{1,0}^{i-h}} < \v{p}^{- \frac{i-h}{p-1}} \times \v{p}^{i-h} = \v{p}^{(1-\frac{1}{p-1})(i-h)}
\ee
Since the image of $n$ is contained in $O_{\Cp}^{\times}$, we have
\be
\v{\rho_{n,i,j}(U)} \ 
\left\{
\begin{array}{ll}
\leq 1 & (i \leq j) \\
< \v{p}^{(1-\frac{1}{p-1})(i-j)} & (i > j)
\end{array}
\right.
\ee
for any $(n,i,j,U) \in W \times \N^2 \times B$. In particular, for any $(i,j,U) \in \N^2 \times B$, the norm of the rigid analytic function $\rho_{n,i,j}(U)$ on $n \in W$ is bounded by $1$ if $i \leq j$ and by $\v{p}^{(1-\frac{1}{p-1})(i-j)}$ otherwise.

\begin{prp}
\label{rho_n is well-defined}
Let $(n,U,v) \in W \times B \times \rC_0(\N,\Cp)$.
\bi
\item[(1)] For any $i \in \N$, the infinite sum $\sum_{j \in \N} \rho_{n,i,j}(U) v(j)$ converges to a $v'_i \in \Cp$ with $\v{v'_i} \leq \n{v}$.
\item[(2)] For any $\epsilon \in \R_{> 0}$, there exists an $i_0 \in \N$ independent of $n$ and $U$ such that $\sup_{i \in \N_{\geq i_0}} \v{v'_i} < \epsilon$.
\ei
\end{prp}

\begin{proof}
The first assertion follows from
\be
\v{\rho_{n,i,j}(U) v(j)} = \v{\rho_{n,i,j}(U)} \ \v{v(j)} \leq \v{v(j)}
\ee
for any $i \in \N$. Let $\epsilon \in \R_{> 0}$. By $v \in \rC_0(\N,\Cp)$, there exists a $j_0 \in \N$ such that $\sup_{j \in \N_{\geq j_0}} \v{v(j)} < \epsilon$. Let $j_1 \in \N$ denote the least integer $j$ satisfying $\v{p}^{(1-\frac{1}{p-1})j} \n{v} < \epsilon$. For any $i \in \N_{\geq j_0 + j_1}$, we have
\be
\v{v'_i} & \leq & \max \ens{\v{\sum_{j \in \N_{< j_0}} \rho_{n,i,j}(U) v(j)},\v{\sum_{j \in \N_{\geq j_0}} \rho_{n,i,j}(U) v(j)}} \\
& \leq & \max \ens{\sup_{j \in \N_{< j_0}} \v{\rho_{n,i,j}(U)} \ \v{v(j)}, \sup_{j \in \N_{\geq j_0}} \v{\rho_{n,i,j}(U)} \ \v{v(j)}} \\
& \leq & \max \ens{\sup_{j \in \N_{< j_0}} \v{p}^{(1-\frac{1}{p-1})(i-j)} \ \n{v}, \sup_{j \in \N_{\geq j_0}} \v{v(j)}} \\
& < & \epsilon.
\ee
\end{proof}

For an $n \in W$, we denote by $\rho_n$ the map
\be
B \times \rC_0(\N,\Cp) & \to & \rC_0(\N,\Cp) \\
(U,v) & \mapsto & \left( \sum_{j \in \N} \rho_{n,i,j}(U) v(j) \right)_{i \in \N},
\ee
which makes sense by Proposition \ref{rho_n is well-defined}.

\begin{prp}
\label{rigid analytic property of rho_n}
Let $(U,v) \in B \times \rC_0(\N,\Cp)$.
\bi
\item[(1)] For any $i \in \N$, $\rho_n(U,v)(i)$ is a rigid analytic function on $n \in W$ of norm $\leq \n{v}$.
\item[(2)] For any $\epsilon \in \R_{> 0}$, there exists an $i_0 \in \N$ independent of $U$ such that for any $n \in W$, the inequality $\sup_{i \in \N_{\geq i_0}} \n{\rho_n(U,v)(i)} < \epsilon$ hold.
\ei
\end{prp}

\begin{proof}
Since $\rho_{n,i,j}(U)$ is a rigid analytic function on $n$ of norm $\leq 1$ for any $(i,j) \in \N^2$, the first assertion immediately follows from Proposition \ref{rho_n is well-defined} (1), and the second assertion immediately follows from Proposition \ref{rho_n is well-defined} (2).
\end{proof}

\begin{prp}
\label{rigid analytic property of the composition}
For any $(U,U',v,i) \in B^2 \times \rC_0(\N,\Cp) \times \N$, $\rho_n(U,\rho_n(U',v))(i)$ is a rigid analytic function on $n \in W$ of norm $\leq \n{v}$.
\end{prp}

\begin{proof}
We have
\be
\rho_n(U,\rho_n(U',v))(i) = \sum_{j \in \N} \rho_{n,i,j}(U) \rho_n(U',v)(j),
\ee
and hence the assertion follows from Proposition \ref{rho_n is well-defined} (1) and Proposition \ref{rigid analytic property of rho_n}.
\end{proof}

\begin{prp}
\label{rho_n is a representation}
For any $n \in W$, $(\rC_0(\N,\Cp),\rho_n)$ is a unitary Banach $\Cp$-linear representation of $B$.
\end{prp}

\begin{proof}
The proof is essentially parallel to that of \cite{Mih15} Proposition 3.11 (i) except for the difference of the coefficient ring. We have
\be
\rho_n \left( \left( \begin{array}{cc} 1 & 0 \\ 0 & 1 \end{array} \right) , v \right) = v
\ee
for any $v \in \rC_0(\N,\Cp)$ by the definition of $\rho_n$. We have $\n{\rho_n(U,v)} \leq \n{v}$ for any $v \in \rC_0(\N,\Cp)$ by Proposition \ref{rigid analytic property of rho_n} (1). Therefore, it suffices to show $\rho_n(U,\rho_n(U',v)) = \rho_n(UU',v)$ for any $(U,U',v) \in B^2 \times \rC_0(\N,\Cp)$. Both hand sides are rigid analytic on $n$ by Proposition \ref{rigid analytic property of rho_n} (1) and Proposition \ref{rigid analytic property of the composition}. Therefore, it is reduced to the case $n_p \in \Zp$ by identity theorem.

\vs
It suffices to show $\v{\rho_n(U,\rho_n(U',v))(i) - \rho_n(UU',v)(i)} < \epsilon$ for any $(i,\epsilon) \in \N \times \R_{> 0}$. By $v \in \rC_0(\N,\Cp)$, there exists a $j_0 \in \N$ such that $\sup_{j \in \N_{\geq j_0}} \v{v(j)} < \epsilon$. We define $v' \in \rC_0(\N,\Cp)$ by
\be
v'(j) \coloneqq 
\left\{
\begin{array}{ll}
v(j) & (j < j_0) \\
0 & (j \geq j_0)
\end{array}
\right..
\ee
By $\n{v - v'} < \epsilon$ and Proposition \ref{rigid analytic property of rho_n}, it suffices to show $\rho_n(U,\rho_n(U',v'))(i) = \rho_n(UU',v')(i)$. For any $n' \in \N_{\geq \max \ens{j_0,i}}$, setting
\be
v'' \coloneqq \sum_{j=0}^{n'} v'(j) X^{n'-j} Y^j \in \Sym^{n'}_{\Cp}(\Cp^2),
\ee
we have
\be
\Sym^{n'}(\rho)(U,\Sym^{n'}(\rho)(U',v''))(i) = \Sym^{n'}(\rho)(UU',v'')(i)
\ee
and hence
\be
\rho_{n'}(U,\rho_n(U',v'))(i) = \rho_{n'}(UU',v')(i)
\ee
by the definition of $\rho_{n'}$ directly interpolating $\Sym^{n'}(\rho)$. Since $n$ is a cluster point of $\N_{\geq \max \ens{j_0,i}} \subset W$ by $n_p \in \Z_p$, we conclude the equality.
\end{proof}

As the proof of Proposition \ref{rho_n is a representation} shows, for any $n \in \N$, $\sum_{i=0}^{n} \Cp \delta_i$ is a non-trivial weakly closed subrepresentation of $(\rC_0(\N,\Cp),\rho_n)$ isomorphic to $(\Sym^n(\Cp^2),\Sym^n(\rho))$, and hence $(\rC_0(\N,\Cp),\rho_n)$ is not weakly irreducible. On the other hand, we obtain the following:

\begin{thm}
\label{weak irreducibility of rho_n}
For any $n \in W$ with $n_p \notin \N$, $(\rC_0(\N,\Cp),\rho_n)$ is weakly irreducible.
\end{thm}

In order to prove Theorem \ref{weak irreducibility of rho_n}, we prepare convention and lemmata. We fix an $n \in W$. We set $M_n \coloneqq (\rC_0(\N,\Cp),\rho_n)^{\vee}$, and identify the underlying $O_{\Cp}$-module of $M_n$ with $O_{\Cp}^{\N}$ through the canonical isomorphism. Then for any $(U,m,i) \in B \times M_n \times \N$, we have
\be
(U m)(i) = \sum_{j \in \N} m(j) \rho_{n,j,i}(U^{-1})
\ee
by the definition of the dual action. Since $O_{\Cp}$ is not a field, it is not straightforward to apply Lie algebra theory. We set
\be
H^{+}_u & \coloneqq & \left( \begin{array}{cc} u & 0 \\ 0 & 1 \end{array} \right) \in B \\
H_u & \coloneqq & \left( \begin{array}{cc} u & 0 \\ 0 & u^{-1} \end{array} \right) \in \SL_2(\Zp) \cap B
\ee
for each $u \in \Zp^{\times}$ and
\be
E & \coloneqq & \left( \begin{array}{cc} 1 & 1 \\ 0 & 1 \end{array} \right) \in \SL_2(\Zp) \cap B \\
F & \coloneqq & \left( \begin{array}{cc} 1 & 0 \\ p & 1 \end{array} \right) \in \SL_2(\Zp) \cap B.
\ee
We set $H^{+} \coloneqq \set{H^{+}_u}{u \in \Zp^{\times}}$ and $H \coloneqq \set{H_u}{u \in \Zp^{\times}}$. For any $(i,j,u) \in \N^2 \times \Zp^{\times}$, we have
\be
\rho_{n,i,j}(H^{+}_u) & = &
\left\{
\begin{array}{ll}
u^{n-i} & (i = j) \\
0 & (i \neq j)
\end{array}
\right.. \\
\rho_{n,i,j}(H_u) & = &
\left\{
\begin{array}{ll}
u^{n-2i} & (i = j) \\
0 & (i \neq j)
\end{array}
\right..
\ee
For any $(i,j,t) \in \N^2 \times \Zp$, we have
\be
\rho_{n,i,j}(E^t) & = & 
\left\{
\begin{array}{ll}
\binom{j}{i} t^{j-i} & (i \leq j) \\
0 & (i > j)
\end{array}
\right. \\
\rho_{n,i,j}(F^t) & = & 
\left\{
\begin{array}{ll}
0 & (i < j) \\
\binom{n-j}{i-j} (pt)^{i-j} & (i \geq j)
\end{array}
\right..
\ee

\begin{lmm}
\label{highest weight}
Let $M \subset M_n$ be a non-zero $F$-stable $O_{\Cp}$-submodule. If $n_p \notin \N$, then there exists an $m \in M$ such that $m(0) \neq 0$.
\end{lmm}

\begin{proof}
We denote by $i_0 \in \N$ the minimum of an $i \in \N$ such that there exists an $m \in M$ with $m(i) \neq 0$, which makes sense by $M \neq \ens{0}$. Assume $i_0 > 0$. Take an $m \in M$ with $m(i_0) \neq 0$. Let $r_0 \in \N$ denote the minimum of an $r \in \N$ such that $\sup_{i \in \N} \v{m(i)} < \v{p^{-r - 1 + \frac{2}{p-1}}(n_p - i_0) m(i_0)}$, which makes sense by $n_p \notin \N$ and $m(i_0) \neq 0$. By the definition of $i_0$, we have
\be
\v{(F^{p^{r_0}} m)(i_0-1)} & = & \v{\sum_{j \in \N} m(j) \rho_{n,j,i_0-1}(F^{-p^{r_0}})} = \v{\sum_{j \in \N_{\geq i_0}} \binom{n-(i_0-1)}{j-(i_0-1)} (p^{r_0+1})^{j-(i_0-1)} m(j)} \\
& = & \v{\sum_{j \in \N_{\geq i_0}} \binom{n-i_0+1}{j-i_0+1} p^{(j-i_0+1)(r_0+1)} m(j)}.
\ee
For any $j \in \N_{\geq i_0}$, we have
\be
\v{m(j) \binom{n-i_0+1}{j-i_0+1} p^{(j-i_0+1)(r_0+1)}} \ 
\left\{
\begin{array}{ll}
= \v{(n_p-i_0) p^{r_0+1} m(i_0)} & (j = i_0) \\
< \v{p^{2(r_0 + 1 - \frac{1}{p-1})} m(j)} & (j = i_0)
\end{array}
\right..
\ee
Therefore, by the definition of $r_0$, we obtain
\be
\v{(F^{p^{r_0}} m)(i_0-1)} = \v{\sum_{j \in \N_{\geq i_0}} \binom{n-i_0+1}{j-i_0+1} p^{(j-i_0+1)(r_0+1)} m(j)} = \v{(n_p-i_0) p^{r_0+1} m(i_0)} \neq 0.
\ee
This contradicts the definition of $i_0$.
\end{proof}

\begin{lmm}
\label{weight zero concentration}
Let $M \subset M_n$ be an $H^{+}$-stable $O_{\Cp}$-submodule. Then for any $(m,i_0) \in M_n \times \N$, there exists an $m' \in i_0! p^{i_0} M$ such that $m'(0) = i_0! p^{i_0} m(0)$ and $m'(i) = 0$ for any $i \in \N_{\leq i_0} \setminus \ens{0}$.
\end{lmm}

\begin{proof}
For each $k \in \N$, we set
\be
\mu_k \coloneqq [H^{+}_{1+p}] - (1+p)^{-n+k} [H^{+}_1] \in O_{\Cp}[H^{+}].
\ee
Since $\set{\mu_k}{k \in \N}$ is commutative, the finite product
\be
\mu \coloneqq \prod_{k=1}^{i_0} \mu_k \in O_{\Cp}[H^{+}]
\ee
makes sense without ambiguity. Set $m' \coloneqq \mu m \in M_n$. For each $i \in \N$, we have
\be
m'(i) = \prod_{k=1}^{i_0} ((1+p)^{-n+i} - (1+p)^{-n+k}) m(i) = (1+p)^{-i_0(n-i)} \prod_{k=1}^{i_0} (1 - (1+p)^{k-i}) m(i).
\ee
This implies the assertion by $\v{1 - (1+p)^h} = \v{ph}$ for any $h \in \N$.
\end{proof}

\begin{lmm}
\label{weight zero spike}
Let $M \subset M_n$ be a weakly closed $H^{+}$-stable pure $O_{\Cp}$-submodule admitting an $m \in M$ with $m(0) \neq 0$. Then $\delta_0 \in M$ holds.
\end{lmm}

\begin{proof}
For any $i_0 \in \N$, there exists an $m' \in M$ such that $m'(0) = m(0)$ and $m'(i) = 0$ for any $i \in \N_{\leq i_0} \setminus \ens{0}$ by Lemma \ref{weight zero concentration} and the purity of $M$. Therefore, we have $m(0) \delta_0 \in M$ by the weak closedness of $M$. Thus, we obtain $\delta_0 \in M$ by the purity of $M$.
\end{proof}

\begin{lmm}
\label{weight shift}
Let $M \subset M_n$ be an $E$-stable $O_{\Cp}$-submodule with $\delta_0 \in M$. Then for any $i_0 \in \N$, there exists a $d \in \N_{> 0}$ such that for any $(a_i)_{i \in \N_{< i_0}} \in d O_{\Cp}^{i_0}$, there exists an $m \in M$ such that $m(i) = a_i$ for any $i \in \N_{< i_0}$.
\end{lmm}

\begin{proof}
For any $(t,i) \in \N^2$, we have
\be
(E^t \delta_0)(i) = \sum_{j \in \N} \delta_0(j) \rho_{n,j,i}(E^{-t}) = \rho_{n,0,i}(E^t) = \binom{i}{0} t^{i-0} = t^i.
\ee
Therefore, the determinant $\prod_{t_0=1}^{i_0-1} \prod_{t_1=t_0+1}^{i_0} (t_1 - t_0)$ of the Vandermonde's matrix associated to $(t)_{t=1}^{i_0}$ satisfies the desired property of $d$.
\end{proof}

\begin{proof}[Proof of Theorem \ref{weak irreducibility of rho_n}]
Let $M \subset M_n$ be a non-zero generically weakly closed pure left $O_{\Cp}[[B]]$-submodule. By Lemma \ref{highest weight}, Lemma \ref{weight zero concentration}, and Lemma \ref{weight zero spike}, we have $\delta_0 \in M$. For any $i_0 \in \N$ and any $(a_i)_{i \in \N_{< i_0}} \in O_{\Cp}^{i_0}$, there exists an $m \in M$ such that $m(i) = a_i$ for any $i \in \N_{< i_0}$ by Lemma \ref{weight shift} and the purity of $M$. This implies the weak density of $M$ in $M_n$, and hence $M = M_n$ by the weak closedness of $M$ in $M_n$. We conclude that $M_n$ is generically weakly simple, and hence $(\rC_0(\N,k),\rho_n)$ is weakly irreducible by Theorem \ref{weak irreducibility criterion}.
\end{proof}

\subsection{Continuous Parabolic Induction of Complex $p$-adic Character}
\label{Continuous Parabolic Induction of Complex p-adic Character}

Next, we construct a $p$-adic family parametrised by $W^2$ in a way similar to one in \cite{ST02} \S 4. For each $U \in B$, we denote by $\nabla_U$ the left multiplication
\be
B & \to & B \\
A & \mapsto & UA.
\ee
We denote by $\pi$ the map
\be
B \times \rC(B,\Cp) & \to & \rC(B,\Cp) \\
(U,f) & \mapsto & f \circ \nabla_{U^{-1}},
\ee
i.e.\ the right regular representation of $B$ over $\Cp$. We denote by $P \subset B$ the closed subgroup of upper triangle regular matrices. We note that although lower triangle matrices are used for the induction in \cite{ST02} \S 4, we use upper triangle matrices in order to reuse the convention in \S \ref{Symmetric Product of Complex p-adic Weight}. The computation does not differ significantly, except that we use LU decomposition rather than Iwahori decomposition and Lazard's theory on divisors for $\Cp[[T]]$ rather than Weierstrass preparation theorem.

\vs
For an $(n,U) \in W^2 \times P$, we set $U^n \coloneqq U_{0,0}^{n(0)} U_{1,1}^{n(1)} \in O_{\Cp}^{\times}$. For an $n \in W^2$, we set
\be
V_n \coloneqq \set{f \in \rC(B,\Cp)}{\forall (A,U) \in B \times P[f(AU) = f(A) U^n]}.
\ee
Then for any $n \in W^2$, $V_n$ is a closed subrepresentation of $(\rC(B,\Cp),\pi)$, and we denote by $\pi_n$ the restriction of $\pi$ to $V_n$. Since $V_n$ is free by \cite{Roo78} Theorem 5.9, $(V_n,\pi_n)$ is a unitary Banach $\Cp$-linear representation of $B$.

\begin{thm}
\label{weak irreducibility of pi_n}
For any $n \in W^2$, if $n(0) - n(1) \notin \N$, then $(V_n,\pi_n)$ is weakly irreducible.
\end{thm}

In order to prove Theorem \ref{weak irreducibility of pi_n}, we prepare convention and lemmata. We fix an $n \in W^2$. We denote by $M_n$ the dual of $(V_n,\pi_n)$.

\vs
We set $L \coloneqq \set{F^t}{t \in \Zp}$. By $\# \cO_L = \aleph_0$, $O_k[[L]]$ is cofree as an $O_k$-module by Proposition \ref{weak topology vs inverse limit topology} (1). We denote by $\pr_L \colon B \twoheadrightarrow L$ and $\pr_P \colon B \twoheadrightarrow P$ the projections with respect to the LU decomposition
\be
B & \stackrel{\cong}{\longrightarrow} & L \times P \\
A & \mapsto & \left( \left( \begin{array}{cc} 1 & 0 \\ A_{0,0}^{-1} A_{1,0} & 1 \end{array} \right) , \left( \begin{array}{cc} A_{0,0} & A_{0,1} \\ 0 & A_{0,0}^{-1} \det(A) \end{array} \right) \right).
\ee
For any $f \in \rC(L,\Cp)$, we denote by $f^n$ the continuous map
\be
B & \to & \Cp \\
A & \mapsto & f(\pr_L(A)) \pr_P(A)^n.
\ee
The restriction map $V_n \to \rC(L,\Cp)$ is an isometric $\Cp$-linear isomorphism, as it is a contracting $\Cp$-linear homomorphism admitting a contracting inverse
\be
\rC(L,\Cp) & \to & V_n \\
f & \mapsto & f^n.
\ee
Through its dual $O_{\Cp}$-linear isomorphism $O_{\Cp}[[L]] \cong M_n$, we regard $O_{\Cp}[[L]]$ as a cofree left $O_{\Cp}[[B]]$-module. For any $(A,t) \in B \times \Zp$, we have
\be
A [F^t] & = & \pr_P(A F^t)^n[\pr_L(A F^t)] \\
& = & (A_{0,0} + pt A_{0,1})^{n(0)} ((A_{0,0} + pt A_{0,1})^{-1} \det(A))^{n(1)} \left[ \pr_L \left( \left( \begin{array}{cc} A_{0,0} + pt A_{0,1} & A_{0,1} \\ A_{1,0} + pt A_{1,1} & A_{1,1} \end{array} \right) \right) \right] \\
& = & (A_{0,0} + pt A_{0,1})^{n(0) - n(1)} \det(A)^{n(1)} \left[ F^{\frac{p^{-1} A_{1,0} + t A_{1,1}}{A_{0,0} + pt A_{0,1}}} \right].
\ee
Through the homeomorphic group isomorphism
\be
\Zp & \to & L \\
t & \mapsto & F^t,
\ee
we obtain an isometric $\Cp$-linear isomorphism $\rC(L,\Cp) \to \rC(\Zp,\Cp)$ and its dual $O_{\Cp}$-linear isomorphism $O_{\Cp}[[\Zp]] \to O_{\Cp}[[L]]$. Through the composite of the isomorphism $O_{\Cp}[[\Zp]] \to O_{\Cp}[[L]]$ and the Iwasawa isomorphism $O_{\Cp}[[\Zp]] \cong O_{\Cp}[[T]]$, we regard $O_{\Cp}[[T]]$ as a cofree left $O_{\Cp}[[B]]$-module. For any $(A,t) \in B \times \Zp$, we have
\be
A (1+T)^t = (A_{0,0} + pt A_{0,1})^{n(0) - n(1)} \det(A)^{n(1)} (1+T)^{\frac{p^{-1} A_{1,0} + t A_{1,1}}{A_{0,0} + pt A_{0,1}}}
\ee
by the definition of the left $O_{\Cp}[[B]]$-module structure of $O_{\Cp}[[T]]$. For an $(i,f) \in \N \times O_{\Cp}[[T]]$, we denote by $c_i(f) \in O_{\Cp}$ the coefficient of $f$ of degree $i$.

\vs
Let $(A,f) \in B \times O_{\Cp}[[T]]$. We have
\be
A f & = & A \sum_{i \in \N} c_i(f) T^i = \sum_{i \in \N} c_i(f) A T^i = \sum_{i \in \N} c_i(f) A((1+T)-1)^i \\
& = & \sum_{i \in \N} c_i(f) A \sum_{j=0}^{i} \binom{i}{j} (-1)^{i-j} (1+T)^j = \sum_{i \in \N} c_i(f) \sum_{j=0}^{i} \binom{i}{j} (-1)^{i-j} A (1+T)^j \\
& = & \sum_{i \in \N} c_i(f) \sum_{j=0}^{i} \binom{i}{j} (-1)^{i-j} (A_{0,0} + pj A_{0,1})^{n(0) - n(1)} \det(A)^{n(1)} (1+T)^{\frac{p^{-1} A_{1,0} + j A_{1,1}}{A_{0,0} + pj A_{0,1}}} \\
& = & \det(A)^{n(1)} \sum_{i \in \N} c_i(f) \sum_{j=0}^{i} \binom{i}{j} (-1)^{i-j} (A_{0,0} + pj A_{0,1})^{n(0) - n(1)} (1+T)^{\frac{p^{-1} A_{1,0} + j A_{1,1}}{A_{0,0} + pj A_{0,1}}}.
\ee
Here, the weak convergence of the sum $\sum_{i \in \N} c_i(f) A T^i$ to $A \sum_{i \in \N} c_i(f) T^i$ abstractly follows from Proposition \ref{characterisation of weak topology} (1) and Lemma \ref{automatic continuity of multiplication}. For the reader's convenience, we show the weak convergence in another concrete way.

\vs
Let $i \in \N$. Set $q_p(i) \coloneqq p^{\lfloor \log_p \max \ens{1,i} \rfloor} \in \N$ and $r_p(i) \coloneqq i - q_p(i) \in \N$. We denote by $\mathfrak{a}_i$ the ideal of $O_{\Cp}[[T]]$ generated by $\set{(p^{-r} q_p(i)) ((1+T)^{p^r} - 1)}{r \in \N_{\leq \lfloor \log_p \max \ens{1,i} \rfloor}}$. We have
\be
T^i & = & T^{r_p(i)+q_p(i)} = ((1+T)-1)^{r_p(i)}(((1+T)-1)^{q_p(i)} - (1-1)^{q_p(i)}) \\
& = & \sum_{j=0}^{r_p(i)} (-1)^{r_p(i)-j} \binom{r_p(i)}{j} (1+T)^j \sum_{k=0}^{q_p(i)} (-1)^{q_p(i)-k} \binom{q_p(i)}{k} ((1+T)^k - 1) \\
& = & \sum_{j=0}^{r_p(i)} \sum_{k=0}^{q_p(i)} (-1)^{r_p(i)-j+q_p(i)-k} \binom{r_p(i)}{j} \binom{q_p(i)}{k} ((1+T)^{j+k} - (1+T)^j) \\
& \in & \mathfrak{a}_i
\ee
by $\binom{p^{h_0}}{p^{h_1}u} \in p^{h_0 - h_1} \Z$ for any $(h_0,h_1,u) \in \N^3$ with $h_0 \geq h_1$ and $\gcd(p,u) = 1$ by Kummer's theorem. Therefore, we obtain
\be
A T^i & = & \sum_{j=0}^{r_p(i)} \sum_{k=0}^{q_p(i)} (-1)^{r_p(i)-j+q_p(i)-k} \binom{r_p(i)}{j} \binom{q_p(i)}{k} (A(1+T)^{j+k} - A(1+T)^j) \\
& = & \sum_{j=0}^{r_p(i)} \sum_{k=0}^{q_p(i)} (-1)^{r_p(i)-j+q_p(i)-k} \binom{r_p(i)}{j} \binom{q_p(i)}{k} \\
& & \left( (A_{0,0} + p(j+k) A_{0,1})^{n(0) - n(1)} \det(A)^{n(1)} (1+T)^{\frac{p^{-1} A_{1,0} + (j+k) A_{1,1}}{A_{0,0} + p(j+k) A_{0,1}}} \right. \\
& & \ \ \left. - (A_{0,0} + pj A_{0,1})^{n(0) - n(1)} \det(A)^{n(1)} (1+T)^{\frac{p^{-1} A_{1,0} + j A_{1,1}}{A_{0,0} + pj A_{0,1}}} \right) \\
& = & \det(A)^{n(1)} \sum_{j=0}^{r_p(i)} \sum_{k=0}^{q_p(i)} (-1)^{r_p(i)-j+q_p(i)-k} \binom{r_p(i)}{j} \binom{q_p(i)}{k} \\
& & \left( (A_{0,0} + p(j+k) A_{0,1})^{n(0) - n(1)} (1+T)^{\frac{p^{-1} A_{1,0} + (j+k) A_{1,1}}{A_{0,0} + p(j+k) A_{0,1}}} \right. \\
& & \ \ \left. - (A_{0,0} + pj A_{0,1})^{n(0) - n(1)} (1+T)^{\frac{p^{-1} A_{1,0} + j A_{1,1}}{A_{0,0} + pj A_{0,1}}} \right) \\
& \in & \mathfrak{a}_i.
\ee
Since $(\mathfrak{a}_i)_{i \in \N}$ forms a fundamental system of neighbourhoods of $0 \in O_{\Cp}[[T]]$ for the weak topology by Proposition \ref{characterisation of weak topology} (1), the sum $\sum_{i \in \N} c_i(f) A T^i$ actually weakly converges to $A \sum_{i \in \N} c_i(f) T^i$.

\vs
For the use in an explicit computation later, we set $\rho_{n,i}(A) \coloneqq \det(A)^{-n(1)} A T^i$ for each $i \in \N$. For any $i \in \N$, we have seen
\be
\rho_{n,i}(A) = \sum_{j=0}^{i} \binom{i}{j} (-1)^{i-j} (A_{0,0} + pj A_{0,1})^{n(0) - n(1)} (1+T)^{\frac{p^{-1} A_{1,0} + j A_{1,1}}{A_{0,0} + pj A_{0,1}}},
\ee
in the first computation of $A \sum_{i \in \N} c_i(f) T^i$ using Lemma \ref{automatic continuity of multiplication}, and $\rho_{n,i}(A) \in \mathfrak{a}_i$ in the second computation of $A \sum_{i \in \N} c_i(f) T^i$ using Kummer's theorem. We have
\be
A f = \det(A)^{n(1)} \sum_{i \in \N} c_i(f) \rho_{n,i}(A).
\ee
In particular, we have
\be
H^{+}_u f & = & u^{n(1)} \sum_{i \in \N} c_i(f) \sum_{j=0}^{i} \binom{i}{j} (-1)^{i-j} u^{n(0) - n(1)} (1+T)^{u^{-1}j} \\
& = & u^{n(0)} \sum_{i \in \N} c_i(f) \left( (1+T)^{u^{-1}} - 1 \right)^i = u^{n(0)} f \left( (1 + T)^{u^{-1}} - 1 \right)
\ee
for any $(u,f) \in \Zp^{\times} \times O_{\Cp}[[T]]$, and
\be
E^t f = \sum_{i \in \N} c_i(f) \sum_{j=0}^{i} \binom{i}{j} (-1)^{i-j} (1 + pjt)^{n(0) - n(1)} (1+T)^{\frac{j}{1 + pjt}}
\ee
for any $(t,f) \in \Zp \times O_{\Cp}[[T]]$.

\vs
We denote by $\Lambda$ the Banach $\Cp$-algebra $\Cp \otimes_{O_{\Cp}} O_{\Cp}[[T]]$ equipped with the gauge $\n{\cdot}_{O_{\Cp}[[T]]}$ (cf.\ Definition \ref{saturated}). See \cite{BGR84} \S 3.7.1 for the notion of a Banach algebra. For a subset $\mathfrak{a} \subset \Lambda$, we set $V(\mathfrak{a}) \coloneqq \set{z \in O_{\Cp}}{\v{z} < 1 \land \forall f \in \mathfrak{a}[f(z) = 0]}$. The weak topology on $O_{\Cp}[[T]]$ coincides with the topology used in \cite{Laz62} by Proposition \ref{characterisation of weak topology} (2), and hence results in \cite{Laz62} are applicable to it.

\begin{lmm}
\label{common zero}
Let $\mathfrak{a} \subset \Lambda$ be an ideal. Then the following are equivalent:
\bi
\item[(1)] The ideal $\mathfrak{a}$ is weakly dense.
\item[(2)] The set $V(\mathfrak{a})$ is empty.
\ei
\end{lmm}

\begin{proof}
First, assume (1). Then there exists a net in $\mathfrak{a}$ weakly convergent to $1 \in O_{\Cp}[[T]]$, and hence every element of $V(\mathfrak{a})$ is a zero of $1$, i.e.\ $V(\mathfrak{a}) = \emptyset$.

\vs
Next, assume (2). Then we have $\mathfrak{a} \neq \ens{0}$. If $\mathfrak{a}$ were not weakly dense in $\Lambda$, then $\mathfrak{a}$ should correspond to a non-trivial divisor by \cite{Laz62} Proposition 10, which would imply $V(\mathfrak{a}) \neq \emptyset$.
\end{proof}

\begin{lmm}
\label{principal}
Let $\mathfrak{a} \subset \Lambda$ be a weakly closed $H^{+}$-stable ideal. Then $1+z$ is a $p$-power root of $1$ for any $z \in V(\mathfrak{a})$, and there exists an $f \in \mathfrak{a}$ such that $V(\ens{f}) = V(\mathfrak{a})$.
\end{lmm}

\begin{proof}
Let $(z,u) \in V(\mathfrak{a}) \times \Zp^{\times}$. For any $f \in \mathfrak{a}$, we have
\be
f((1 + z)^u - 1) = u^{n(0)} (H^{+}_{u^{-1}} f)(z) = 0
\ee
by the $H^{+}$-stability of $\mathfrak{a}$. This implies $(1 + z)^u - 1 \in V(\mathfrak{a})$. On the other hand, $V(\mathfrak{a})$ is countable by by the property of Newton polygon. This implies that $\set{(1+z)^u-1}{u \in \Zp^{\times}}$ is a finite set, i.e.\ $1+z$ is a $p$-power root of $1$.

\vs
For each $r \in \N$, we denote by $d_r \in \N$ the maximum of a $d \in \N$ such that there is a $z \in \Cp$ such that $1+z$ is a primitive $p^r$-th root of $1$ and for any $f \in \mathfrak{a} \setminus \ens{0}$, the multiplicity of zero of $f$ at $z$ is not less than $d$. We denote by $D$ the formal product of the divisors associated to $T^{d_0} \Zp[T]$ and $(\frac{(1+T)^{p^{r+1}}-1}{p((1+T)^{p^{r}}-1)})^{d_{r+1}} \in 1 + T \Zp[T]$ with $r \in \N$ (cf.\ \cite{Laz62} (4.3) and (4.9)). Since each component of $D$ is defined by an element of $\Zp[T]$, there exists an $f \in \Zp[[T]]$ corresponding to $D$ by \cite{Laz62} Theorem 1. By the construction, $D$ corresponds to $\mathfrak{a}$ by the correspondence in \cite{Laz62} Proposition 10. Therefore, we have $f \in \mathfrak{a}$ by weak closedness of $\mathfrak{a}$ and \cite{Laz62} Proposition 10.
\end{proof}

\begin{lmm}
\label{convolutive Vandermonde}
Let $k$ be a complete valuation field of characteristic $0$. For any $(c,\epsilon) \in \ell^{\infty}(\N,k) \times O_k$ with $\v{\epsilon} < 1$, if the infinite sum
\be
\sum_{i \in \N} c(i) \sum_{j=0}^{i} (-1)^{i-j} \binom{i}{j} j^h (1 + \epsilon)^j
\ee
converges to $0$ in $k$ for any $h \in \N$, then $c = 0$ holds.
\end{lmm}

\begin{proof}
We denote by $f \in k[[z]]$ the generating function of $c$. By $c \in \ell^{\infty}(\N,k)$, $f$ is convergent on $\set{z \in O_k}{\v{z} < 1}$. We denote by $D$ the differential operator $(1+z) \frac{d}{dz}$ on $k[[z]]$. For any $h \in \N$, we have
\be
D^h f & = & D^h \sum_{i \in \N} c(i) z^i = D^h \sum_{i \in \N} c(i) ((1+z)-1)^i = D^h \sum_{i \in \N} c(i) \sum_{j=0}^{i} (-1)^{i-j} \binom{i}{j} (1+z)^j \\ 
& = & \sum_{i \in \N} c(i) \sum_{j=0}^{i} (-1)^{i-j} \binom{i}{j} j^h (1+z)^j,
\ee
and hence $(D^h f)(\epsilon) = 0$. Assume $c \neq 0$. We denote by $h_0 \in \N$ the multiplicity of zero of $f$ at $\epsilon$, which makes sense by $c \neq 0$. Then the multiplicity of zero of $D^{h_0} f$ at $\epsilon$ is $0$ by $1 + \epsilon \neq 0$, which contradicts $(D^{h_0} f)(\epsilon) = 0$.
\end{proof}

\begin{lmm}
\label{p-power root unstability}
Let $f \in \Lambda$. If there is a $z_0 \in V(\set{E^t f}{t \in \Zp})$ such that $1+z_0$ is a $p$-power root of $1$, then either $(n(0) - n(1))_p \in \N$ or $f = 0$ holds.
\end{lmm}

\begin{proof}
Take an $r \in \N$ with $(1+z_0)^{p^r} = 1$. For any $t \in \Zp$, we have
\be
0 = (E^t f)(z_0) = \sum_{i \in \N} c_i(f) \sum_{j=0}^{i} \binom{i}{j} (-1)^{i-j} (1 + pjt)^{n(0) - n(1)} (1+z_0)^{\frac{j}{1 + pjt}}.
\ee
Set $s \coloneqq (n(0) - n(1))_p \in O_{\Cp}$. Substituting $t = p^r z$ with $z \in \Zp$ to the equation above, we obtain
\be
0 & = & \sum_{i \in \N} c_i(f) \sum_{j=0}^{i} \binom{i}{j} (-1)^{i-j} (1 + p^{r+1}jz)^{n(0) - n(1)} (1+z_0)^{\frac{j}{1 + p^{r+1}jz}} \\
& = & \sum_{i \in \N} c_i(f) \sum_{j=0}^{i} \binom{i}{j} (-1)^{i-j} (1 + p^{r+1}jz)^s (1+z_0)^j.
\ee
Let $i \in \N$. We denote by $g_i(z) \in \Cp \ens{z}$ the Taylor expansion of the function
\be
\sum_{j=0}^{i} \binom{i}{j} (-1)^{i-j} (1 + p^{r+1}jz)^s (1+z_0)^j.
\ee
appeared in the sum above, and by $\mathfrak{b}_i$ the ideal of $O_{\Cp}$ generated by
\be
\set{(p^{-k} q_p(i))((1+z_0)^{p^k} - 1)}{k \in \N_{\leq \lfloor \log_p \max \ens{1,i} \rfloor}}.
\ee
For any $z \in \Zp$, we have $g_i(z) = \rho_{n,i}(E^{p^r z}) |_{T = z_0}$ by the definition of $g_i$ and $\rho_{n,i}$, and also
\be
& & \rho_{n,i}(E^{p^r z}) |_{T = z_0} \\
& = & \det(E^{p^r z})^{n(1)} \sum_{j=0}^{r_p(i)} \sum_{k=0}^{q_p(i)} (-1)^{r_p(i)-j+q_p(i)-k} \binom{r_p(i)}{j} \binom{q_p(i)}{k} \\
& & \left( (1 + p(j+k) p^r z)^{n(0) - n(1)} (1+z_0)^{\frac{p^{-1} \cdot 0 + (j+k) \cdot 1}{1 + p(j+k) \cdot p^r z}} - (1 + pj \cdot p^r z)^{n(0) - n(1)} (1+z_0)^{\frac{p^{-1} \cdot 0 + j \cdot 1}{1 + pj \cdot p^r z}} \right) \\
& = & \sum_{j=0}^{r_p(i)} \sum_{k=0}^{q_p(i)} (-1)^{r_p(i)-j+q_p(i)-k} \binom{r_p(i)}{j} \binom{q_p(i)}{k} \\
& & \left( (1 + p^{r+1}(j+k)z)^s (1+z_0)^{j+k} - (1 + p^{r+1}jz)^s (1+z_0)^j \right) \\
& \in & \mathfrak{b}_i O_{\Cp} \ens{z}
\ee
by the argument in the second computation of $A \sum_{i \in \N} c_i(f) T^i$ for an $A \in B$ using Kummer's theorem. By $\mathfrak{b}_i \subset p^{-r+1} q_p(i) O_{\Cp}$, we obtain $\n{g_i} \leq \v{p}^{-r + 1 + \lfloor \log_p \max \ens{1,i} \rfloor}$. This implies that the sum $\sum_{i \in \N} c_i(f) g_i$ converges to a $g \in \Cp \ens{z}$. We have $g(z) = \sum_{i \in \N} c_i(f) g_i(z) = 0$ for any $z \in \Zp$, and hence $g = 0$ by identity theorem. By the norm convergence of the sum defining $g$, higher derivatives of $g = 0$ are computable by termwise derivation. Therefore, for any $h \in \N$, we have
\be
0 & = & \frac{d^h g}{dz^h} = \sum_{i \in \N} c_i(f) \frac{d^h g_i}{dz^h} \\
& = & \sum_{i \in \N} c_i(f) \frac{d^h}{dz^h} \sum_{j=0}^{i} \binom{i}{j} (-1)^{i-j} (1 + p^{r+1}jz)^s (1+z_0)^j \\
& = & \sum_{i \in \N} c_i(f) \sum_{j=0}^{i} \binom{i}{j} (-1)^{i-j} \binom{s}{h} (p^{r+1}j)^h (1 + p^{r+1}jz)^{s-h} (1+z_0)^j \\
& = & \binom{s}{h} p^{(r+1)h} \sum_{i \in \N} c_i(f) \sum_{j=0}^{i} \binom{i}{j} (-1)^{i-j} j^h (1 + p^{r+1}jz)^{s-h} (1+z_0)^j,
\ee
and hence we obtain
\be
0 = \binom{s}{h} p^{(r+1)h} \sum_{i \in \N} c_i(f) \sum_{j=0}^{i} \binom{i}{j} (-1)^{i-j} j^h (1+z_0)^j,
\ee
by substituting $z = 0$. Assume $s \notin \N$. Then for any $h \in \N$, we have $\binom{s}{h} \neq 0$ and hence
\be
\sum_{i \in \N} c_i(f) \sum_{j=0}^{i} \binom{i}{j} (-1)^{i-j} j^h (1+z_0)^j = 0.
\ee
By Lemma \ref{convolutive Vandermonde}, we obtain $f = 0$.
\end{proof}

\begin{proof}[Proof of Theorem \ref{weak irreducibility of pi_n}]
Let $M \subset M_n$ be a generically weakly closed pure left $O_{\Cp}[[B]]$-submodule. We denote by $\mathfrak{a} \subset O_{\Cp}[[T]]$ the ideal corresponding to $M$ by the $B$-equivariant $O_{\Cp}$-linear isomorphism $O_{\Cp}[[T]] \cong M_n$. Since $\mathfrak{a} \subset O_{\Cp}[[T]]$ is generically weakly closed, $\Cp \otimes_{O_{\Cp}} \mathfrak{a}$ is weakly closed in $\Lambda$. Therefore, there exists an $f \in \mathfrak{a}$ such that $V(\ens{f}) = V(\mathfrak{a})$ and every zero of $f$ is a $p$-power root of $1$ by Lemma \ref{principal}.

\vs
First, suppose $V(\mathfrak{a}) = \emptyset$. Then we have $\Cp \otimes_{O_{\Cp}} \mathfrak{a} = \Lambda$ by generic weak closedness of $\mathfrak{a}$ and \cite{Laz62} Proposition 10, and hence $\mathfrak{a} = O_{\Cp}[[T]]$ by the purity of $\mathfrak{a}$. Therefore, we obtain $M = M_n$.

\vs
Next, suppose $V(\mathfrak{a}) \neq \emptyset$. Then we have $f = 0$ by $(n(0) - n(1))_p \notin \N$ and Lemma \ref{p-power root unstability}, and hence $V(\mathfrak{a}) = V(\ens{f}) = \set{z \in O_{\Cp}}{\v{z} < 1}$. This implies $\Cp \otimes_{O_{\Cp}} \mathfrak{a} = \ens{0}$ by \cite{Laz62} Proposition 10, and hence $\mathfrak{a} = \ens{0}$. Therefore, we obtain $M = \ens{0}$.

\vs
We have obtained either $M = \ens{0}$ or $M = M_n$. We conclude that $M_n$ is generically weakly simple, and hence $(V_n,\pi_n)$ is weakly irreducible by Theorem \ref{weak irreducibility criterion}.
\end{proof}

% \newpage
\vspace{0.3in}
\addcontentsline{toc}{section}{Acknowledgements}
\noindent {\Large \bf Acknowledgements}
\vspace{0.2in}

\noindent
I thank K.\ Eda for informing me of Specker phenomenon. I thank T.\ Hayashi for recalling me of representation theory. I thank K.\ Tokimoto for helping me to consider applications of Banach representations over $\Cp$. I thank K.\ Ishizuka for answering references of preceding studies. I thank all people who helped me to learn mathematics and programming. I also thank my family.


\addcontentsline{toc}{section}{References}
\begin{thebibliography}{99}

\bibitem[BGR84]{BGR84} S.\ Bosch, U.\ G\"untzer, and R.\ Remmert, {\it Non-Archimedean Analysis A Systematic Approach to Rigid Analytic Geometry}, Grundlehren der mathematischen Wissenschaften 261, A Series of Comprehensive Studies in Mathematics, Springer, 1984.

\bibitem[Bla92]{Bla92} A.\ Blass, {\it Cardinal Characteristics and the Product of Countably Many Infinite Cyclic Groups}, 	arXiv:9209203, 1992.

\bibitem[Dal00]{Dal00} H.\ G.\ Dales, {\it Banach Algebra and Automatic Continuity}, London Mathematical Society Monographs New Series, Volume 24, Oxford Science Publications, 2000.

\bibitem[Eda82]{Eda82} K.\ Eda, {\it A Boolean Power and a Direct Product of Abelian Groups}, Tsukuba Journal of Mathematics, Volume 6, Number 2, pp.\ 187--194, 1982.

\bibitem[Eda83]{Eda83} K.\ Eda, {\it A note on subgroups of $\Z^{\N}$}, Abelian group theory (R.\ G\"obel and E.\ Walker,eds.), Lecture Notes in Mathematics, Volume 1006, pp.\ 371--374, Springer, 1983.

\bibitem[E{\L}54]{EL54} A.\ Ehrenfeucht and J.\ {\L}o\'s, {\it Sur le produits cart\'esiens des groupes cycliques infinis}, Bulletin de l'Academie Polonaise des Sciences, Volume III, Num\'ero 2, pp.\ 261--263, 1954.

\bibitem[Laz62]{Laz62} M.\ Lazard, {\it Les z\'eros d'une fonction analytique d'une variable sur un corps valu\'e complet}, Publications math\'ematiques de l'I.H.\'E.S., Tome 14, pp.\ 47--75, 1962.

\bibitem[Mih14]{Mih14} T.\ Mihara, {\it Characterisation of the Berkovich Spectrum of the Banach Algebra of Bounded Continuous Functions}, Documenta Mathematica, Volume 19, pp.\ 769--799, 2014.

\bibitem[Mih15]{Mih15} T.\ Mihara, {\it Galois Representations Associated to $p$-adic Families of Modular Forms of Finite Slope}, arXiv:1504.04728, 2015.

\bibitem[Mih21-1]{Mih21-1} T.\ Mihara, {\it Duality theory of $p$-adic Hopf algebras}, Categories and General Algebraic Structures with Applications, Volume 14, Issue 1, pp.\ 81--118, 2021.

\bibitem[Mih21-2]{Mih21-2} T.\ Mihara, {\it Schneider--Teitelbaum duality for locally profinite groups}, Categories and General Algebraic Structures with Applications, Volume 14, Issue 1, pp.\ 119--166, 2021.

\bibitem[MN89]{MN89} K.\ Morita and J.\ Nagata, {\it Topics in general topology}, North-Holland mathematical library, Volume 41, North-Holland, 1989.

\bibitem[Roo78]{Roo78} A.\ C.\ M.\ van Rooij, {\it Non-Archimedean Functional Analysis}, Marcel Dekker, 1978.

\bibitem[Sch84]{Sch84} W.\ H.\ Schikhof, {\it Locally Convex Spaces over Nonspherically Complete Valued Fields}, Groupe d'\'etude d'Analyse ultram\'etrique, 12e ann\'ee, Number 24, pp.\ 1--33, 1984--1985.

\bibitem[Sch95]{Sch95} W.\ H.\ Schikhof, {\it A Perfect Duality between $p$-adic Banach Spaces and Compactoids}, Indagationes Mathematicae, Volume 6, Issue 3, pp.\ 325--339, 1995.

\bibitem[Sch99]{Sch99} W.\ H.\ Schikhof, {\it Banach spaces over nonarchimedean valued fields}, Report Number 9937, Radboud University Nijmegen (formerly Katholieke Universiteit Nijmegen), 1999.

\bibitem[Sch02]{Sch02} P.\ Schneider, {\it Nonarchimedean Functional Analysis}, Springer Monographs in Mathematics, Springer, 2002.

\bibitem[Spe50]{Spe50} E.\ Specker, {\it Additive gruppen von Folgen ganzer Zahlen}, Portugaliae Mathematica, Volume 9, pp.\ 131--140, 1950.

\bibitem[ST02]{ST02} P.\ Schneider and J.\ Teitelbaum, {\it Banach space representations and Iwasawa theory}, Israel Journal of Mathematics, Volume 127, Issue 1, pp.\ 359--380, 2002.

\bibitem[Woo93]{Woo93} W.\ H.\ Woodin, {\it A Discontinuous Homomorphism from $C(X)$ without CH}, Journal of the London Mathematical Society, Volume s2-48 (2), p.\ 299-315, 1993.

\bibitem[Zee55]{Zee55} E.\ C.\ Zeeman, {\it On direct sums of free cycles}, Journal of the London Mathematical Society, Volume s1-30, Issue 2, pp.\ 195--212, 1955.

\end{thebibliography}
\end{document}